\documentclass[10pt]{article}
\usepackage[a4paper, total={7in, 9in}]{geometry}
\usepackage{graphicx}
\usepackage{cancel}
\usepackage{authblk}
\usepackage{fancyhdr}
\usepackage{graphicx}
\usepackage{enumerate}
\usepackage{mathtools}
\usepackage{multicol}
\usepackage{amsmath,amssymb}
\usepackage{amsthm}
\usepackage[utf8]{inputenc}
\usepackage[ruled,vlined]{algorithm2e}
\usepackage{subfig}
\usepackage{booktabs}
\usepackage{multirow}
\usepackage[colorlinks=true]{hyperref}

\usepackage{hyperref}

\usepackage{enumitem}

\newcounter{hypo}

\newtheorem{theorem}{Theorem}
\newtheorem{proposition}[theorem]{Proposition}
\newtheorem{lemma}[theorem]{Lemma}
\theoremstyle{definition}
\newtheorem{example}[theorem]{Example}
\newtheorem{remark}[theorem]{Remark}
\newtheorem{hypothesis}{Hypothesis}

\newcommand{\R}{\mathbb{R}}
\newcommand{\N}{\mathbb{N}}
\newcommand{\HH}{\mathcal{H}}
\newcommand\inner[2]{\langle #1, #2 \rangle}

\DeclareMathOperator{\Fix}{Fix}
\DeclareMathOperator{\Zer}{Zer}
\DeclareMathOperator{\prox}{prox}

\DeclareMathOperator{\dist}{dist}

\newcommand{\norm}[1]{\left\lVert#1\right\rVert}
\newcommand{\normc}[1]{\left\lVert#1\right\rVert^2}

\newcommand{\paren}[1]{\left(#1\right)}
\newcommand{\parenc}[1]{\left[#1\right]}
\newcommand{\pint}[2]{\left<#1,#2\right>}
\newcommand{\cali}{\mathcal}

\newcommand{\wto}{\rightharpoonup}

\newcommand{\Rcupinf}{\mathbb{R}\cup\lbrace+\infty\rbrace}

\title{Inertial Krasnoselskii-Mann Iterations}
\author{Juan José Maulén\thanks{Bernoulli Institute for Mathematics, Computer Science and Artificial Intelligence, University of Groningen \& Center for Mathematical Modelling, Mathematical Engineering Department, University of Chile} \and Ignacio Fierro\thanks{BIOCORE team, Centre INRIA de l’Universit\'e de la C\^ote d’Azur} \and
	Juan Peypouquet\thanks{Bernoulli Institute for Mathematics, Computer Science and Artificial Intelligence, University of Groningen}
}

\begin{document}
	
	\maketitle
	
	\begin{abstract}
	
	We establish the weak convergence of inertial Krasnoselskii-Mann iterations towards a common fixed point of a family of quasi-nonexpansive operators, along with estimates for the non-asymptotic rate at which the residuals vanish. Strong and linear convergence are obtained in the quasi-contractive setting. In both cases, we highlight the relationship with the non-inertial case, and show that passing from one regime to the other is a continuous process in terms of the hypotheses on the parameters. Numerical illustrations are provided for an inertial primal-dual method and an inertial three-operator splitting algorithm, whose performance is superior to that of their non-inertial counterparts.  \\

	

	
\end{abstract}

	\section{Introduction}

Krasnoselskii-Mann (KM) iterations \cite{Krasnosel1955two,mann1953mean} are at the core of numerical methods used in optimization, fixed point theory and variational analysis, since they include many fundamental splitting algorithms whose convergence can be analyzed in a unified manner. These include the {\it forward-backward} \cite{LionsMercier,Passty} to approximate a zero of the sum of two maximally monotone operators, and its various particular instances: on the one hand, we have the {\it gradient projection} algorithm \cite{goldstein1964convex,levitin1966constrained}, the gradient method \cite{cauchy1847methode} and the proximal point algorithm
\cite{martinet1970regularisation,rockafellar1976monotone,brezis1978produits,guler1991convergence}, to cite some abstract methods, as well as the {\it Iterative Shrinkage-Thresholding Algorithm} (ISTA) \cite{daubechies2004iterative,combettes2005signal}, to speak more concretely. KM iterations also encompass other splitting methods like  {\it Douglas-Rachford} \cite{DouglasRachford1956}, primal-dual methods \cite{chambolle2011first,attouch2010parallel,combettes2004solving,vu2013splitting,condat2013primal} and the three-operator splitting \cite{davis2017three}.  \\


In convex optimization, first order methods can be enhanced by adding an inertial substep, motivated by physical considerations \cite{polyakInertial,Nesterov1983AMF,alvarez_minimizing_2000}. To our knowledge, the first extensions beyond the optimization setting was developed in \cite{alvarez2001inertial}, followed by \cite{mainge2008convergence,Lorenz2014,moudafi2003convergence}  some years later. The main drawback of the previous results is that they require an implicit hypothesis on the sequence generated by the algorithm (the summability of a certain series) to ensure its convergence. In \cite{alvarez2001inertial}, however, this difficulty is overcome, in some special cases and for small values of the inertial parameters. These ideas were also used in \cite{boct2015inertial}, and then improved in \cite{dong2021new}, by adapting the inertial factors to the relaxation ones (see below). A similar principle had been used in \cite{AttouchCabot2019}, whose analysis was based on \cite{attouch2019convergence}. Nonasymptotic convergence rates for the residuals have been given in \cite{shehu2018convergence,iyiola2021new}. Strong and linear convergence can be found in \cite{shehu2020inertial}, for strictly contractive {\it forward-projection} operators. Other extensions have been considered in \cite{dong2015accelerated,combettes2017quasi,moudafi2018reflected,dong2018general}. See also \cite{dong2022krasnosel} for a more thorough account of KM iterations, with and without inertia. Interest in this type of methods increased remarkably in the past decade in view of theoretical advances in the convergence theory for the {\it Fast Iterative Shrinkage-Thresholding Algorithm} (FISTA) \cite{beck2009fast}, obtained in \cite{ChambolleDossal2015,AttPeyRedMAPR2018,AttPey1_k2}. \\

The purpose of this work is to develop further insight into the convergence properties of {\it inertial Krasnoselskii-Mann iterations} in their general form
\begin{equation}\label{E:Algorithm}
	\left\{
	\begin{array}{rcl}
		y_k & = & x_k+\alpha_k(x_k-x_{k-1})\\
		x_{k+1} & = & (1-\lambda_k)y_k+\lambda_kT_ky_k,
	\end{array}
	\right.
\end{equation} 
where $(T_k)$ is a family of operators defined on a real Hilbert space $\HH$, and the positive sequences $(\alpha_k)$ and $(\lambda_k)$ are the {\it inertial} and {\it relaxation} (or {\it averaging}) parameters, respectively. 

\begin{remark}
	To fix the ideas, suppose $\inf_{k\ge 1}\lambda_k>0$, $(\alpha_k)$ is bounded, and $T_k\equiv T$, where $T$ is continuous. If $x_k$ happens to converge to a point $\bar x$, then the {\it residual} $\|Tx_k-x_k\|$ goes to zero, and $\bar x$ is a fixed point of $T$.
\end{remark}

Our general aim is to provide conditions on the parameter sequences and the family of operators to ensure that the sequences generated by \eqref{E:Algorithm} converges (weakly or strongly) to a common fixed point of the $T_k$'s, provided there are any. More specifically, we mean to establish a setting, which is as general as possible, but such that (1) the hypotheses are interpretable and verifiable; (2) the proofs are transparent and mostly elementary; and (3) the convergence results are quantifiable in terms of appropriate sequences. We shall also see that adding the inertial term does not always make algorithms faster (this is reflected in the worst-case convergence rates), but may boost their convergence in some relevant instances. Another interesting line of research consists in identifying the combination of parameters for which the algorithm has its best numerical performance. Although we consider this highly relevant, we shall not pursue that direction here. \\

The paper is organized as follows: in Section \ref{section:weak} we establish  the weak convergence of the iterations towards a common fixed point of the family of operators in the quasi-nonexpansive case, along with a non-asymptotic rate at which the residuals vanish. Section \ref{section:strong} is devoted to the strong and linear convergence in the quasi-contractive setting. In both cases, we highlight the relationship with the non-inertial case, and show that passing from one regime to the other is a continuous process in terms of parameter hypotheses and convergence rates. In Section \ref{section:examples}, we discuss several instances of KM iterations, which are relevant to the numerical illustrations provided in Section \ref{section:numerical}, concerning an inertial primal-dual method and an inertial three-operator splitting algorithm.  

	\section{Vanishing residuals and weak convergence}\label{section:weak}

An operator $T:\HH\to \HH$ is {\it quasi-nonexpansive} if $\Fix(T)\neq\emptyset$ and $\|Ty-p\|\le\|y-p\|$ for all $y\in \HH$ and $p\in\Fix(T)$. This implies, in particular, that
\begin{equation} \label{E:quasi-cocoercive}
	2\pint{y-p}{Ty-y}\le-\|Ty-y\|^2
\end{equation} 
for all $y\in \HH$ and $p\in\Fix(T)$. \\

In this section, we consider a family $(T_k)$ of quasi-nonexpansive operators on $\HH$, with $F:=\bigcap_{k\ge 1}\Fix(T_k)\neq\emptyset$, along with a sequence $(x_k,y_k)$ satisfying \eqref{E:Algorithm}, where $(\alpha_k)$ is a nondecreasing sequence\footnote{This is just to simplify the proof and is sufficiently general for practical purposes.} in $[0,1)$, and $(\lambda_k)$ is a sequence in $(0,1)$ such that $\inf_{k\ge 1}\lambda_k>0$. \\

To simplify the notation, given $p\in F$, we set
\begin{equation} \label{E:def_C_Delta_delta}
	\left\{\begin{array}{ccl} 
		\nu_k & = & \left(\lambda_{k}^{-1}-1\right)
		\smallskip  \\
		\delta_k & = & \nu_{k-1}(1-\alpha_{k-1})\|x_k-x_{k-1}\|^2, \smallskip  \\
		\Delta_k(p) & = & \|x_k-p\|^2-\|x_{k-1}-p\|^2,\quad\Delta_1(p)=0 \smallskip \\
		C_k(p) & = & \|x_k-p\|^2-\alpha_{k-1}\|x_{k-1}-p\|^2+\delta_k,\quad C_1(p)=\|x_1-p\|^2.
	\end{array} \right.
\end{equation} 

At different points, and in order to simplify the computations, we shall make use of a basic property of the norm in $\HH$: for every $x,y\in\HH$ and $\alpha \in [0,1]$, we have 
\begin{equation}\label{eq:id_new}
	\norm{\alpha x + (1-\alpha)y}^2 = \alpha \norm{x}^2 + (1-\alpha)\norm{y}^2 - \alpha(1-\alpha)\norm{x-y}^2.    
\end{equation}

The following auxiliary result will be useful in the sequel:

\begin{lemma}\label{L:1}
	Let $(T_k)$ be a family of quasi-nonexpansive operators on $\HH$, with $F:=\bigcap_{k\ge 1}\Fix(T_k)\neq\emptyset$, and let $(x_k,y_k)$ satisfy \eqref{E:Algorithm}. For each $k\ge 1$ and $p\in F$, we have
	\begin{equation} \label{E:proof5}
		\Delta_{k+1}(p)+\delta_{k+1}+\nu_k\alpha_k\|x_{k+1}-2x_k+x_{k-1}\|^2 \le \alpha_k\Delta_k(p) + \big[\alpha_k(1+\alpha_k)+\nu_k\alpha_k(1-\alpha_k)\big]\|x_k-x_{k-1}\|^2.
	\end{equation}
\end{lemma}

\begin{proof}
	Take $p\in F$. From \eqref{E:Algorithm}, it follows that
	\begin{equation} \label{E:proof1}
		\|x_{k+1}-p\|^2 = \|y_k-p\|^2 +\lambda_k^2\|y_k-T_ky_k\|^2+2\lambda_k\pint{y_k-p}{T_ky_k-y_k} \le \|y_k-p\|^2 -\lambda_k(1-\lambda_k)\|y_k-T_ky_k\|^2,
	\end{equation} 
	where the inequality is given by \eqref{E:quasi-cocoercive}. Notice that
	$$\|y_k-p\|^2 = \|(1+\alpha_k)(x_k-p) - \alpha_k(x_{k-1}-p)\|^2, $$
	and using \eqref{eq:id_new} we get
	\begin{equation} \label{E:proof2}
		\|y_k-p\|^2 = (1+\alpha_k)\|x_k-p\|^2+\alpha_k(1+\alpha_k)\|x_k-x_{k-1}\|^2-\alpha_k\|x_{k-1}-p\|^2.
	\end{equation} 
	By combining expressions \eqref{E:proof1} and \eqref{E:proof2}, we obtain
	$$
	\|x_{k+1}-p\|^2 \le (1+\alpha_k)\|x_k-p\|^2+\alpha_k(1+\alpha_k)\|x_k-x_{k-1}\|^2 -\alpha_{k}\|x_{k-1}-p\|^2-\lambda_k(1-\lambda_k)\|y_k-T_ky_k\|^2.
	$$
	Recalling from \eqref{E:def_C_Delta_delta} that $\Delta_k(p)=\|x_k-p\|^2-\|x_{k-1}-p\|^2$, we rewrite the latter as
	\begin{equation} \label{E:proof3}
		\Delta_{k+1}(p) \le \alpha_k\Delta_k(p) + \alpha_k(1+\alpha_k)\|x_k-x_{k-1}\|^2 -\lambda_k(1-\lambda_k)\|y_k-T_ky_k\|^2.
	\end{equation}
	Notice that 
	\begin{equation} \label{E:proof3.5}
		\lambda_k^2\|y_k-T_ky_k\|^2 = \|x_{k+1}-x_k - \alpha_k(x_k - x_{k-1})\|^2 = \norm{(1-\alpha_k)(x_{k+1} - x_k) + \alpha_k (x_{k+1}-2x_k + x_{k-1})}^2,
	\end{equation}
	and using \eqref{eq:id_new} gives
	\begin{equation} \label{E:proof3.5bis}
		\lambda_k^2\|y_k-T_ky_k\|^2 = (1-\alpha_k)\|x_{k+1}-x_k\|^2-\alpha_k(1-\alpha_k)\|x_k-x_{k-1}\|^2+\alpha_k\|x_{k+1}-2x_k+x_{k-1}\|^2.
	\end{equation}
	By multiplying the latter by $\nu_k=(1-\lambda_k)/\lambda_k$, and using the definition of $\delta_k$ in \eqref{E:def_C_Delta_delta}, we rewrite this as
	\begin{equation} \label{E:proof34}
		\delta_{k+1}+\nu_k\alpha_k\|x_{k+1}-2x_k+x_{k-1}\|^2=\nu_k\alpha_k(1-\alpha_k)\|x_k-x_{k-1}\|^2 +\lambda_k(1-\lambda_k)\|y_k-T_ky_k\|^2.
	\end{equation}
	Summing \eqref{E:proof3} and \eqref{E:proof34}, we obtain \eqref{E:proof5}.
\end{proof}

We are now in a position to show that the sequence $(x_n)$ remains anchored to the set $F$, while both the residuals $\|y_k-T_ky_k\|$ and the speed $\|x_k-x_{k-1}\|$ tend to 0. We shall make some assumptions on the parameter sequences $(\alpha_k)$ and $(\lambda_k)$.

\begin{hypothesis} \label{H:1}
	There is $k_0$ such that $\alpha_k(1+\alpha_k)+(\lambda_k^{-1}-1)\alpha_k(1-\alpha_k)-(\lambda_{k-1}^{-1}-1)(1-\alpha_{k-1}) \le 0$ for all $k\ge k_0$.
\end{hypothesis}

A reinforced version with strict inequality is given by:

\begin{hypothesis} \label{H:2}
	$\limsup_{k\to\infty}\big[\alpha_k(1+\alpha_k)+(\lambda_k^{-1}-1)\alpha_k(1-\alpha_k)-(\lambda_{k-1}^{-1}-1)(1-\alpha_{k-1})\big]<0$.
\end{hypothesis}

\begin{remark} \label{R:hypotheses}
	With Hypothesis \ref{H:1} or \ref{H:2}, there exist $\varepsilon\ge 0$ and $k_0\ge 1$ such that	
	\begin{equation} \label{E:H:1_bis}
		\alpha_k(1+\alpha_k)+(\lambda_k^{-1}-1)\alpha_k(1-\alpha_k) \le (\lambda_{k-1}^{-1}-1)(1-\alpha_{k-1})-\varepsilon
	\end{equation}	
	for all $k\ge k_0$ (if Hypothesis \ref{H:2} holds, then $\varepsilon>0$; otherwise, $\varepsilon=0$). Also, under Hypothesis \ref{H:2}, $\alpha:=\sup_{k\ge 1}\alpha_k<1$ and $\lambda:=\inf_{k\ge 1}\lambda_k>0$.
\end{remark}

\begin{theorem} \label{T:1}
	Let $(T_k)$ be a family of quasi-nonexpansive operators on $\HH$, and let $(x_k,y_k)$ satisfy \eqref{E:Algorithm}. Suppose that the set $F=\bigcap_{k\ge 1}\Fix(T_k)$ is nonempty.
	\begin{itemize}
		\item [i)] If Hypothesis \ref{H:1} holds, for every $p\in F$, the sequence $\big(C_k(p)\big)_{k\ge k_0}$ is nonincreasing and nonnegative, thus $\lim\limits_{k\to\infty}C_k(p)$ exists.
		\item [ii)] If Hypothesis \ref{H:2} holds, the series $\sum\limits_{k\ge 1}\|x_{k+1}-2x_k+x_{k-1}\|^2$, $\sum\limits_{k\ge 1}\|x_k-x_{k-1}\|^2$, $\sum\limits_{k\ge 1}\delta_k$  and $\sum\limits_{k\ge 1}\|y_k-T_ky_k\|^2$ are convergent, and there is a constant $M>0$, depending only on $(\alpha_k)$ and $(\lambda_k)$, such that
		\begin{equation} \label{E:nonasymptotic}
			\min\limits_{1\le k\le n}\|y_k-T_ky_k\|^2\le\frac{M\dist(x_1,F)^2}{n}.    
		\end{equation}
		Moreover, for each $p\in F$, $\lim\limits_{k\to\infty}\|x_k-p\|$ exists.
	\end{itemize}
\end{theorem}

\begin{proof}
	Without any loss of generality, we may assume that \eqref{E:H:1_bis} holds with $k_0=1$. Take any $p\in F$, and combine \eqref{E:H:1_bis} with \eqref{E:proof5}, to obtain
	\begin{eqnarray} 
		\Delta_{k+1}(p)+\delta_{k+1}+\nu_k\alpha_k\|x_{k+1}-2x_k+x_{k-1}\|^2 & \le & \alpha_k\Delta_k(p) + \big[\nu_{k-1}(1-\alpha_{k-1})-\varepsilon\big]\|x_k-x_{k-1}\|^2 \nonumber \\
		& = & \alpha_k\Delta_k(p) +\delta_k-\varepsilon\|x_k-x_{k-1}\|^2. \label{E:proof6}
	\end{eqnarray}
	On the one hand, \eqref{E:proof6} immediately gives 
	\begin{equation} \label{E:L:1_1}
		\Delta_{k+1}(p)  \le  \alpha_k\Delta_k(p)+\delta_k. 
	\end{equation}
	On the other, since $(\alpha_k)$ is nondecreasing, we have 
	\begin{eqnarray*}
		C_{k+1}(p)-C_k(p) = \Delta_{k+1}(p)-\big(\alpha_k\|x_k-p\|^2-\alpha_{k-1}\|x_{k-1}-p\|^2\big) +\delta_{k+1}-\delta_k \le \Delta_{k+1}(p)+\delta_{k+1}-\alpha_k\Delta_k(p) -\delta_k.
	\end{eqnarray*}
	Therefore, \eqref{E:proof6} implies
	\begin{equation} \label{E:L:1_2}
		C_{k+1}(p)+\nu_k\alpha_k\|x_{k+1}-2x_k+x_{k-1}\|^2+\varepsilon\|x_k-x_{k-1}\|^2 \le C_k(p).
	\end{equation}
	It ensues that $\big(C_k(p)\big)$ is nonincreasing. To show that it is nonnegative, suppose that $C_{k_1}(p)<0$ for some $k_1\ge 1$. Since $\big(C_k(p)\big)$ is nonincreasing, 
	$$\|x_k-p\|^2-\alpha_{k-1}\|x_{k-1}-p\|^2\le C_k(p)\le C_{k_1}(p)<0$$ 
	for all $k\ge k_1$. If follows that $\|x_k-p\|^2 \le \|x_{k-1}-p\|^2+ C_{k_1}(p)$, and so
	$$0\le \|x_k-p\|^2\le\|x_{k-1}-p\|^2+C_{k_1}(p) \le \dots \le \|x_{k_1}-p\|^2+(k-k_1)C_{k_1}(p)$$
	for all $k\ge k_1$, which is impossible. As a consequence $\big(C_k(p)\big)$ is nonnegative, and $\lim\limits_{k\to\infty}C_k(p)$ exists. \\
	For ii), Inequality \eqref{E:H:1_bis} holds with $\varepsilon>0$. The summability of the first two series follows from \eqref{E:L:1_2}. In particular,
	\begin{equation} \label{E:L:1_2bis}
		\varepsilon\sum_{k\ge 1}\|x_k-x_{k-1}\|^2 \le C_1(p)=\|x_1-p\|^2.
	\end{equation}
	The third one is a consequence of the second one, since $\lambda:=\inf_{k\ge 1}\lambda_k>0$. 
	For the last one, use \eqref{E:proof3.5bis} to write
	$$\lambda_k^2\|y_k-T_ky_k\|^2\le (1+\alpha)\|x_{k+1}-x_k\|^2+\alpha(1+\alpha)\|x_k-x_{k-1}\|^2.$$
	In view of \eqref{E:L:1_2bis}, this gives the summability of the fourth series, with
	$$n\min\limits_{1\le k\le n}\|y_k-T_ky_k\|^2\le\sum_{k\ge 1}\|y_k-T_ky_k\|^2 \le \frac{(1+\alpha)^2}{\varepsilon\lambda^2}\|x_1-p\|^2.$$
	Since this holds for each $p\in F$, we obtain 
	\eqref{E:nonasymptotic} with $M=\frac{(1+\alpha)^2}{\varepsilon\lambda^2}$. Now, denoting the positive part of $d\in\R$ by $[d]_+$, we obtain from \eqref{E:L:1_1} that
	$$(1-\alpha)\big[\Delta_{k+1}(p)\big]_++\alpha\big[\Delta_{k+1}(p)\big]_+ \le \alpha\big[\Delta_k(p)\big]_++\delta_k.$$
	Summing for $k\ge 1$, we obtain
	$$(1-\alpha)\sum_{k\ge 1}\big[\Delta_{k+1}(p)\big]_+\le \alpha\big[\Delta_1(p)\big]_++\sum_{k\ge 1}\delta_k=\sum_{k\ge 1}\delta_k<\infty.$$

	By writing $h_k=\|x_k-p\|^2-\sum_{j=1}^k\big[\Delta_{j}(p)\big]_+$, we get $h_{k+1}-h_k=\Delta_{k+1}(p)-\big[\Delta_{k+1}(p)\big]_+\le 0$, from which we conclude that $\lim\limits_{k\to\infty}\|x_k-p\|=\lim\limits_{k\to\infty}h_k$ exists. 

\end{proof}

\begin{remark} \label{R:1}
	Hypotheses \ref{H:1} and \ref{H:2} are closely related, but different, from the hypotheses used in \cite{AttouchCabot2019} for forward-backward iterations. In the non-inertial case $\alpha=0$, Hypothesis \ref{H:1} is just $\limsup_{k\to\infty}\lambda_k<1$. On the other hand, since $(\alpha_k)$ is nondecreasing and bounded, we have $\alpha_k\to\alpha\in[0,1]$. If $\lambda_k\to \lambda$, then Hypothesis \ref{H:2} is reduced to
	\begin{equation} \label{E:H2_constant}
		\lambda(1-\alpha+2\alpha^2)<(1-\alpha)^2.
	\end{equation} 
	For each $\alpha\in[0,1)$, there is $\lambda_\alpha>0$ such that \eqref{E:H2_constant} holds for all $\lambda<\lambda_\alpha$.
\end{remark}

In order to prove the weak convergence of the sequences generated by Algorithm \eqref{E:Algorithm}, we shall use the following nonautonomous extension of the concept of demiclosedness. \\

The family of operators $(I-T_k)$ is {\it asymptotically demiclosed at $0$} if for every sequence $(z_k)$ such that $z_k\wto z$ and $z_k-T_kz_k\to 0$, we must have $z\in F=\bigcap_{k\ge 1}\Fix(T_k)$. \\

Of course, if $T:\HH\to \HH$ is nonexpansive and $T_k\equiv T$, then $I-T_k$ is asymptotically demiclosed at $0$. We shall discuss other examples in the next section.

\begin{theorem} \label{T:1bis}
	Let $(T_k)$ be a family of quasi-nonexpansive operators on $\HH$, with $F=\bigcap_{k\ge 1}\Fix(T_k)\neq\emptyset$. Let $(x_k,y_k)$ satisfy \eqref{E:Algorithm}, and assume Hypotheses \ref{H:2} holds. If $(I-T_k)$ is asymptotically demiclosed at $0$, then both $x_k$ and $y_k$ converge weakly, as $k\to\infty$, to a point in $F$.
\end{theorem}

\begin{proof}
	Recall that $\lim\limits_{k\to\infty}\|y_k-T_ky_k\|=\lim\limits_{k\to\infty}\|x_k-x_{k-1}\|=0$, by part ii) of Theorem \ref{T:1}. From \eqref{E:Algorithm}, we deduce that $(y_k)$ and $(x_k)$ have the same (weak and strong) limit points. Suppose $x_{n_k}\wto x$. Then, $y_{n_k}\wto x$ as well. Since $y_{n_k}-T_ky_{n_k}\to 0$, the asymptotic demiclosedness implies $x\in F$. Opial's Lemma \cite{opial1967weak} (see, for instance, \cite[Lemma 5.2]{peypouquet2015convex}) yields the conclusion. 
\end{proof}

\section{Strong and linear convergence}\label{section:strong}

We now focus on the strong convergence of the sequences generated by \eqref{E:Algorithm}, and their convergence rate. As before, we assume that $(\alpha_k)$ is nondecreasing but we do not assume, in principle, that $\inf_{k\ge 1}\lambda_k>0$. \\

Given $q\in(0,1)$, an operator $T:\HH\to \HH$ is $q$-{\it quasi-contractive} if $\Fix(T)\neq\emptyset$ and $\|Ty-p\|\le q\|y-p\|$ for all $y\in \HH$ and $p\in\Fix(T)$. If $T$ is $q$-quasi-contractive, then $\Fix(T)=\{p^*\}$.

Given $\lambda,q\in(0,1)$ and $\xi\in[0,1]$, we define    \begin{equation} \label{E:Q}
	Q(\lambda,q,\xi):=\xi\big(1-\lambda +\lambda q^2\big)+(1-\xi)(1-\lambda+\lambda q)^2=(1-\lambda+\lambda q)^2+\xi\lambda(1-\lambda)(1-q)^2.
\end{equation}

Notice that $Q(\lambda,q,\xi)\in (0,1)$, and that it decreases as $\lambda$ increases, or as either $q$ or $\xi$ decreases. The quantity $Q(\lambda,q,\xi)$ will play a crucial role in the linear convergence rate of the sequences satisfying \eqref{E:Algorithm}. The inclusion of the auxiliary parameter $\xi$ will also allow us to establish convergence rates, with and without inertia, in a unified manner (see the discussion in Subsection \ref{SS:remarks}).	\\

The following result establishes a bound on the distance to a solution after performing a standard KM step:

\begin{lemma}
	Let $T:\HH\to \HH$ be $q$-quasi-contractive with fixed point $p^*$, and let $x,y\in \HH$ and $\lambda>0$ be such that $x=(1-\lambda)y+\lambda Ty$. Then, for each $\xi\in[0,1]$, we have
	\begin{equation} \label{E:Lemma2_1}
		\|x-p^*\|^2 \le Q(\lambda,q,\xi)\|y-p^*\|^2-\xi\lambda(1-\lambda)\|Ty-y\|^2.
	\end{equation}
\end{lemma}	

\begin{proof}
	Notice that 
	$$\norm{x-p^*} = \norm{(1-\lambda)(y-p^*) + \lambda(Ty - p^*)}.$$
	Then, using \eqref{eq:id_new}, we get
	\begin{eqnarray} \label{E:Lemma2_P1}
		\|x-p^*\|^2 & = & (1-\lambda)\|y-p^*\|^2+\lambda\|Ty-p^*\|^2-\lambda(1-\lambda)\|Ty-y\|^2 \nonumber\\
		& \le &  \big(1-\lambda +\lambda q^2\big)\|y-p^*\|^2-\lambda(1-\lambda)\|Ty-y\|^2.
	\end{eqnarray}
	On the other hand, we have
	\begin{equation} \label{E:Lemma2_P2}
		\|x-p^*\|\le(1-\lambda)\|y-p^*\|+\lambda\|Ty-p^*\|\le (1-\lambda+\lambda q)\|y-p^*\|.
	\end{equation}
	Then, inequality \eqref{E:Lemma2_1} is just a convex combination of \eqref{E:Lemma2_P1} and the square of \eqref{E:Lemma2_P2}.
\end{proof}	

\subsection{Convergence analysis}

We now turn to the convergence of the sequences verifying \eqref{E:Algorithm}.	To simplify the notation, for each $k\in\N$, we set
$$\tilde C_k(p)  =  \|x_k-p^*\|^2-\alpha_{k-1}\|x_{k-1}-p^*\|^2+\xi\delta_k\quad\hbox{with}\quad \tilde C_1(p^*)=\|x_1-p^*\|^2.$$

We have the following:

\begin{proposition}
	Let $(T_k)$ be a sequence of operators on $\HH$, such that $\Fix(T_k)\equiv \{p^*\}$ and $T_k$ is $q_k$-quasi-contractive for each $k\in\N$. Let $(x_k,y_k)$ satisfy \eqref{E:Algorithm}, and let $\xi\in [0,1]$. Write $Q_k=Q(\lambda_k,q_k,\xi)$, where $Q$ is defined in \eqref{E:Q}. For each $k\in\N$, we have
	\begin{eqnarray} \label{E:Lemma3}
		\|x_{k+1}-p^*\|^2 +\xi\delta_{k+1} & \le & Q_k\left[ (1+\alpha_k)\|x_k-p^*\|^2-\alpha_k\|x_{k-1}-p^*\|^2\right] \nonumber \medskip \\
		&&\, +\big[Q_k\alpha_k(1+\alpha_k)+\xi\nu_k\alpha_k(1-\alpha_k)\big]\|x_k-x_{k-1}\|^2.
	\end{eqnarray}  
	If, moreover,
	\begin{equation}\label{E:H:2}
		Q_k\alpha_k(1+\alpha_k)+\xi\nu_k\alpha_k(1-\alpha_k)-\xi Q_k\nu_{k-1}(1-\alpha_{k-1})\le 0
	\end{equation}
	for all $k\in\N$, then
	\begin{equation} \label{E:Lemma3_1}
		\tilde C_{k+1}(p^*)\le \left[\prod_{j=1}^{k}Q_j\right]\|x_1-p^*\|^2
	\end{equation}
	and
	\begin{equation} \label{E:Lemma3_2}
		\|x_{k+1}-p^*\|^2 \le \left[\alpha^k
		+ \sum_{j=1}^k\alpha^{k-j}\left[\prod_{i=1}^{j}Q_i\right]
		\right]\|x_1-p^*\|^2.
	\end{equation}
\end{proposition}

\begin{proof}
	We use \eqref{E:Algorithm} and \eqref{E:Lemma2_1} to obtain
	$$\|x_{k+1}-p^*\|^2 \leq Q_k\|y_k-p^*\|^2 -\xi\lambda_k(1-\lambda_k)\|y_k-T_ky_k\|^2.$$
	
	Now, by \eqref{E:proof2}, we deduce that
	$$\|x_{k+1}-p^*\|^2 \le Q_k\left[ (1+\alpha_k)\|x_k-p^*\|^2+\alpha_k(1+\alpha_k)\|x_k-x_{k-1}\|^2-\alpha_k\|x_{k-1}-p^*\|^2\right]-\xi\lambda_k(1-\lambda_k)\|y_k-T_ky_k\|^2.$$
	On the other hand, from \eqref{E:proof34}, we get
	$$\xi\delta_{k+1} \le \xi \nu_k\alpha_k(1-\alpha_k)\|x_k-x_{k-1}\|^2 +\xi \lambda_k(1-\lambda_k)\|y_k-T_ky_k\|^2,$$
	and the last two inequalities together imply \eqref{E:Lemma3}. For the second part, inequalities \eqref{E:Lemma3} and \eqref{E:H:2} together give
	$$\|x_{k+1}-p^*\|^2 +\xi\delta_{k+1} \le Q_k\left[ (1+\alpha_k)\|x_k-p^*\|^2-\alpha_k\|x_{k-1}-p^*\|^2\right]+\xi Q_k\delta_k.$$	
	Subtracting $\alpha_k\|x_k-p^*\|^2$, we are left with
	\begin{eqnarray*} 
		\tilde C_{k+1}(p^*) & \le & \big(Q_k(1+\alpha_k)-\alpha_k\big)\|x_k-p^*\|^2-\alpha_k Q_k\|x_{k-1}-p^*\|^2+\xi Q_k \delta_k \\
		& \le & Q_k\|x_k-p^*\|^2-Q_k\alpha_{k-1}\|x_{k-1}-p^*\|^2+\xi Q_k\delta_k \\
		& = & Q_k\tilde C_k(p^*),
	\end{eqnarray*} 
	where the second inequality comes from $\alpha_k$ being nondecreasing and $Q_k \leq 1$. This gives \eqref{E:Lemma3_1}, recalling that $\tilde C_1(p^*)=\|x_1-p^*\|^2$.
	Now, since $\|x_{k+1}-p^*\|^2-\alpha_{k}\|x_{k}-p^*\|^2\le \tilde C_{k+1}(p^*)$, we have 
	
	$$\|x_{k+1}-p^*\|^2 \le \alpha_{k}\|x_{k}-p^*\|^2+\left[\prod_{j=1}^{k}Q_j\right]\|x_1-p^*\|^2 \le \alpha\|x_{k}-p^*\|^2+\left[\prod_{j=1}^{k}Q_j\right]\|x_1-p^*\|^2,$$
	which we then iterate to obtain \eqref{E:Lemma3_2}.
\end{proof}

The preceding estimations allow us to establish the main result of this section, namely:

\begin{theorem} \label{T:2}
	Let $(T_k)$ be a sequence of operators on $\HH$, such that $\Fix(T_k)\equiv \{p^*\}$ and $T_k$ is $q_k$-quasi-contractive for each $k\in\N$. Let $(x_k,y_k)$ satisfy \eqref{E:Algorithm}, and let $\xi\in [0,1]$. Write $Q_k=Q(\lambda_k,q_k,\xi)$, and assume that \eqref{E:H:2} holds for all $k\in\N$. We have the following:
	\begin{itemize}
		\item [i)] If $\sum_{k=1}^\infty\lambda_k(1-q_k^2)=\infty$, then $x_k$ converges strongly to $p^*$, as $k\to\infty$.
		\item [ii)] If $\lambda_k\ge\lambda>0$ and $q_k\le q<1$ for all $k\in\N$, then $x_k$ converges linearly to $p^*$, as $k\to\infty$. More precisely, 
		\begin{equation} \label{E:rate}
			\|x_k-p^*\|^2 \le
			\left[ \frac{Q(\lambda,q,\xi)^{k+1}-\alpha^{k+1}}{Q(\lambda,q,\xi)-\alpha}\right]\|x_1-p^*\|^2=\mathcal O\left(Q(\lambda,q,\xi)^{k}\right). 
		\end{equation} 
	\end{itemize}
	\if{ 	$$\|x_{k+1}-p^*\|^2 \le \alpha^k\left[1+ \sum_{j=1}^k\left[\frac{1-\lambda(1-q^2)}{\alpha}\right]^j\right]\|x_1-p^*\|^2,$$
		where $\lambda=\inf_{k\ge 1}\lambda_k$. If we set $\beta_+=\max\{\alpha,1-\lambda(1-q^2)\}$ and $\beta_-=\min\{\alpha,1-\lambda(1-q^2)\}$, then
		\begin{equation} \label{E:rate}
			\|x_k-p^*\|^2\le \frac{2}{\beta_+-\beta_-}\left(\beta_+^k-\beta_-^k\right)\|x_1-p^*\|^2=\mathcal O\left(\beta_+^k\right).
		\end{equation} 
	}\fi 
\end{theorem}

\begin{proof}
	For part i), write $p_k=\lambda_k(1-q_k^2)$, and observe that $Q_k\le 1-p_k$, because $Q$ increases with $\xi$. It ensues that
	$$\prod_{k=1}^K Q_k\le\prod_{k=1}^K(1-p_k)=\exp\left[\sum_{k=1}^K\ln(1-p_k)\right]\le \exp\left[-\sum_{k=1}^K p_k\right]$$
	since $\ln(1-z)\le-z$. If $\sum_{k=1}^\infty\lambda_k(1-q_k^2)=\infty$, then $\prod_{k=1}^{\infty}Q_k=0$. By \eqref{E:Lemma3_1}, $\lim_{k\to\infty}\tilde C_k(p^*)=0$. As in the proof of Theorem \ref{T:1}, we can show that the sum of the first two terms in $\tilde C_k(p^*)$, namely $\|x_k-p^*\|^2-\alpha_{k-1}\|x_{k-1}-p^*\|^2$, is nonnegative. Therefore, $\lim_{k\to\infty}\left[\|x_k-p^*\|^2-\alpha_{k-1}\|x_{k-1}-p^*\|^2\right]=0$. If $\alpha_k\equiv 0$, the conclusion is straightforward. Otherwise, given any $\varepsilon>0$, there is $K\in\N$ such that 
	$$\|x_k-p^*\|^2\le\alpha \|x_{k-1}-p^*\|^2+\varepsilon$$
	for all $k\ge K$, since $\alpha_k$ is nondecreasing. This implies
	$$\|x_k-p^*\|^2\le\alpha^{k-K}\|x_{K}-p^*\|^2+\varepsilon(1-\alpha)^{-1},$$
	so that $\limsup_{k\to\infty}\|x_k-p^*\|\le\varepsilon(1-\alpha)^{-1}$, and the conclusion follows. \\
	For ii), we know that $Q(\lambda_k,q_k,\xi)\le Q(\lambda,q,\xi)$, because $Q$ increases either if $\lambda$ decreases, and also if $q$ increases. Gathering the common factors in the second and third terms on the left-hand side of inequality \eqref{E:H:2}, we deduce that $Q\ge \alpha$ (strictly if $\alpha>0$). Using \eqref{E:Lemma3_2}, and observing that the case $Q(\lambda,q,\xi)=\alpha$ is incompatible with inequality \eqref{E:H:2}, we deduce that
	$$\|x_{k+1}-p^*\|^2 \le \alpha^k\left[ \sum_{j=0}^k\left(\frac{Q(\lambda,q,\xi)}{\alpha}\right)^j\right]\|x_1-p^*\|^2=\left[ \frac{\alpha^{k+1}-Q(\lambda,q,\xi)^{k+1}}{\alpha-Q(\lambda,q,\xi)}\right]\|x_1-p^*\|^2,$$
	as claimed.
\end{proof}

\subsection{Behavior with and without inertia}

In the non-inertial case $\alpha_k\equiv 0$, \eqref{E:H:2} holds if either $\xi=0$ or $\lambda_k\le 1$ for all $k$, as in Hypothesis \ref{H:1}. This is less restrictive than Hypothesis \ref{H:2} (see Remark \ref{R:1}). To simplify the explanation, suppose $q_k\equiv q\in (0,1)$. The best convergence rate is
$$\|x_k-p^*\|=\mathcal O\big(q^k\big),$$
obtained from Theorem \ref{T:2} with $\lambda_k\equiv 1$ and $\xi=0$.
If $\alpha_k>0$ for at least one $k$, the case $\xi=0$ is ruled out, and
$$q^2\le(1-\lambda_k+\lambda_k q)^2= Q(\lambda_k,q,0)\le Q(\lambda_k,q,\xi)\le Q(\lambda_k,q,1)= 1-\lambda_k+\lambda_k q^2.$$
All inequalities are strict if $\lambda_k\in (0,1)$. This suggests that there may be operators for which the inertial step actually deteriorate the convergence, so inertial steps should be handled with caution and this can be seen as an argument {\it against} the use of inertia. Actually, it is possible to find a wide variety of behaviors, {\it even for some of the simplest operators}, as shown by the following case study:

\begin{example} \label{EG:minus_I}
	Let $\lambda_k\equiv \lambda\in(0,1)$ and $\alpha_k\equiv\alpha\in[0,1)$. Take $q\in(0,1]$, and consider the operator $T:\R\to \R$, defined by $Ty=-qy$, whose unique fixed point is the origin. \\
	
	If $\alpha=0$, for each $k\ge 0$, we have $x_{k+1}=Lx_k$, where we have written $L=1-\lambda(1+q)$. Iterating from $x_0=1$, we obtain $|x_k|=|L|^k$. If $\lambda(1+q)=1$, convergence occurs in one iteration. \\
	
	Now, let $\alpha\in(0,1)$, so that \eqref{E:Algorithm} reads
	\begin{equation}\label{E:Algorithm_minusI}
		x_{k+1}=L\big(x_k+\alpha(x_k-x_{k-1})\big). 	
	\end{equation}
	Here, we take $x_1=x_0=1$. We can rewrite \eqref{E:Algorithm_minusI} in matrix form as
	$$X_{k+1}=M X_k,\qquad\hbox{where}\qquad 
	M=\left(\begin{array}{cc}(1+\alpha)L & -\alpha L\\ 1 & 0\end{array}\right)\qquad\hbox{and}\qquad X_k=\left(\begin{array}{c}x_k \\ x_{k-1}\end{array}\right).$$
	As before, convergence occurs in one step if $L=0$. The eigenvalues of $M$ are
	$$\mu_{\pm}=\frac{(1+\alpha)L\pm \sqrt{(1+\alpha)^2L^2-4\alpha L}}{2}.$$
	Let us consider the case $L>0$ first. 
	If $(1+\alpha)^2L^2<4\alpha L$ (which is $\lambda(1+q)>(1-\alpha)^2/(1+\alpha)^2$), the eigenvalues are complex conjugates, both with modulus $|\mu_{\pm}|=\sqrt{\alpha L}<1$. Now, $\sqrt{\alpha L}<L$ if, and only if, $L>\alpha$, which means that $\lambda(1+q)<1-\alpha$. 
	Since $|x_k|=\mathcal O(|\mu_{\pm}|)$, the inertial iterations converge strictly faster than the noninertial ones if
	$$\frac{(1-\alpha)^2}{(1+\alpha)^2}<\lambda(1+q)<1-\alpha.$$
	If $L=\alpha$, the convergence rate is the same.
	Else, if $(1+\alpha)^2L^2\ge 4\alpha L$, then $M$ has two real eigenvalues (counting multiplicities), with $0<\mu_-\le \mu_+$. But since
	$L\in(0,1)$ implies $-L<-L^2$, we always have
	$$\mu_+<\frac{(1+\alpha)L+ \sqrt{(1+\alpha)^2L^2-4\alpha L^2}}{2}=\frac{(1+\alpha)L+L\sqrt{(1-\alpha)^2}}{2}=L<1.$$
	Therefore, the inertial iterations also converge strictly faster if
	$$0<\lambda(1+q)\le \frac{(1-\alpha)^2}{(1+\alpha)^2}.$$
	When $L<0$ ($\lambda(1+q)>1$), the matrix $M$ will always have two real eigenvalues, one of each sign. It is easy to verify that $|\mu_+|<|\mu_-|$, which implies that $|\mu_-|$ determines the convergence (the initial condition is not an eigenvector of $M$, so both eigenvalues intervene). But
	$$\mu_-=-\frac{(1+\alpha)|L|+ \sqrt{(1+\alpha)^2L^2+4\alpha |L|}}{2}<-\frac{(1+\alpha)|L|+ \sqrt{(1+\alpha)^2L^2}}{2}=-|L|=L.$$
	In this case, the inertial algorithm performs worse than the noninertial one. Moreover, the inertial iterations do not converge if $\mu_-\le -1$, which is equivalent to
	$$\lambda(1+q)\ge\frac{2(1+\alpha)}{1+2\alpha}.$$
	A few comments are in order:
	\begin{itemize}
		\item For $0<\lambda(1+q)<1-\alpha$, the inertial iterations converge at a strictly faster linear rate than the noninertial ones, even in the noncontracting case $q=1$. 
		\item At the transition point $\lambda(1+q)=1-\alpha$ the convergence rate is the same. 
		\item In the interval $1-\alpha<\lambda(1+q)<\frac{2(1+\alpha)}{1+2\alpha}$, the inertial step is counterproductive and noninertial iterations perform better, except for the singular value $\lambda(1+q)=1$, where both converge in one iteration. In both cases, the closer $\lambda(1+q)$ is to $1$, the faster the convergence.
		\item If $\lambda(1+q)\ge \frac{2(1+\alpha)}{1+2\alpha}$, the inertial iterations do not converge, while the noninertial ones do. This combination of parameters is not feasible if $q\le 1/3$. Notice that, picking $\lambda$ and $\alpha$ satisfying \eqref{E:H2_constant} can be read as picking $\lambda < S(\alpha)$, with $S(\alpha) = \frac{(1-\alpha)^2}{1-\alpha+2\alpha^2}$. Calling $P(\alpha)=\frac{1+\alpha}{1+2\alpha}$, it is easy to see that 
		\[\lambda  < S(\alpha) < P(\alpha) \leq \dfrac{2}{1+q} P(\alpha), \; \forall q \in (0,1].\]    
		Then $\lambda(1+q)<\frac{2(1+\alpha)}{1+2\alpha}$ for all $q \in (0,1]$. Therefore, this last case is incompatible with Hypotheses \eqref{H:1} or \eqref{H:2}.
	\end{itemize}
\end{example}

Now, the convergence rate results given by Theorem \ref{T:2} correspond to worst-case scenarios, which certainly must include cases like the one discussed in Example \ref{EG:minus_I}. However, this situation need not be representative of other concrete instances found in practice, in which inertia improves either the theoretical convergence rate guarantees (see Subsection \ref{S:Gradient} below, and the commented references), or the actual behavior when the algorithm is implemented. In fact, the numerical tests reported below show noticeable improvements in the performance of the selected algorithms, upon adding the inertial substep.

\subsection{Some insights into inequality \eqref{E:H:2}}\label{SS:remarks}

To fix the ideas, we comment on some special cases of inequality \eqref{E:H:2}, especially with constant parameters:

\begin{enumerate}
	\item In the limiting case $q_k\equiv 1$, we have $Q_k\equiv 1$. With constant parameters $\lambda_k\equiv \lambda$, $\alpha_k\equiv\alpha$, \eqref{E:H:2} becomes
	$$\lambda\alpha(1+\alpha)-\xi(1-\lambda)(1-\alpha)^2\le 0.$$
	If 
	\begin{equation} \label{E:feasibility} \frac{\alpha\lambda(1+\alpha)}{(1-\lambda)(1-\alpha)^2}\le  1,
	\end{equation}
	then, there is $\xi_{\alpha,\lambda,1}\in(0,1)$ such that \eqref{E:H:2} holds for all $\xi\in[\xi_{\alpha,\lambda,1},1]$. If $\xi=1$, it is precisely the constant case in Hypothesis \ref{H:1} (see \eqref{E:H2_constant} for a more direct comparison).
	
	\item Keeping $\lambda_k\equiv \lambda\in(0,1)$, $\alpha_k\equiv\alpha\in(0,1)$, and fixing $\xi=1$, let us take $q_k\equiv q\in(0,1)$. In this case, condition \eqref{E:H:2} is equivalent to
	\begin{equation} \label{E:H_q_constant}
		\Psi(\lambda):=(1+\alpha^2)(1-q^2)\lambda^2 -\big(2\alpha^2+(1-\alpha)(2-q^2)\big)\lambda +(1-\alpha)^2\ge 0.
	\end{equation} 
	Observe that $\Psi(0)=(1-\alpha)^2>0$, while $\Psi(1) = -\alpha q^2(1+\alpha)<0$. Since $\Psi$ is quadratic, the equation $\Psi(\lambda)=0$ has exactly one root in $(0,1)$, which we denote by $\lambda_{\alpha,q}$. It follows that, for each $(\alpha,q)\in[0,1)\times(0,1)$, inequality \eqref{E:H_q_constant} holds for all $\lambda\le\lambda_{\alpha,q}$. The values of $\lambda_{\alpha,q}$ on $[0,1)\times(0,1)$ are depicted in Figure \ref{fig:heatmap_lambdas}. 
	\begin{figure}[htbp]
		\centering
		\includegraphics[scale=0.3]{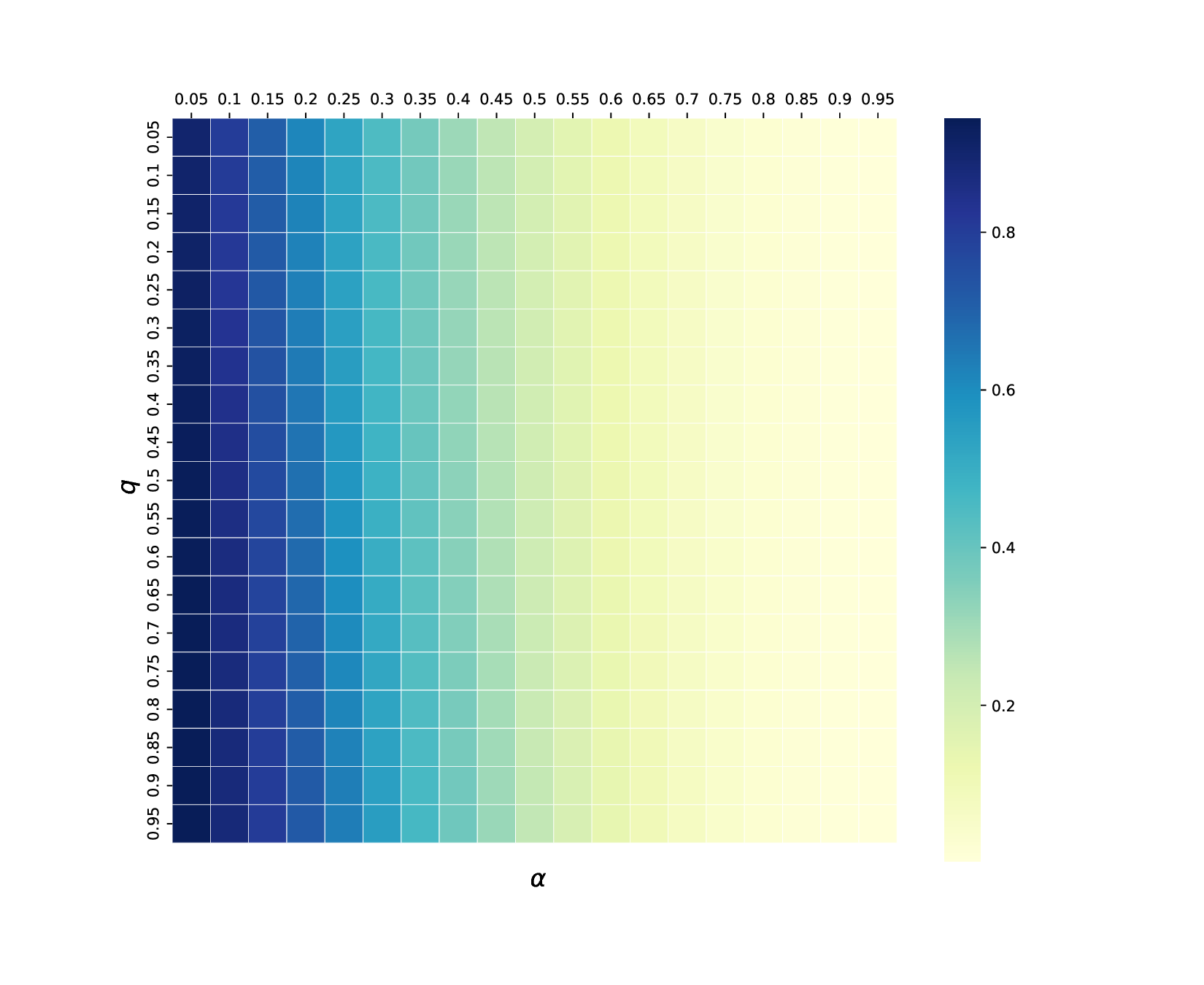}
		\caption{Values of $\lambda_{\alpha,q}$.  }
		\label{fig:heatmap_lambdas}
	\end{figure}
	Once a value for the inertial parameter $\alpha$ has been selected, the best theoretical convergence rate is
	$$Q(\lambda_{\alpha,q},q,1)=1-\lambda_{\alpha,q}(1-q^2).$$
	On the other hand, using the formula for the roots of a quadratic equation and some algebraic manipulations, we deduce that
	$$\left[\frac{2\alpha^2+(1-\alpha)}{2\alpha^2+(1-\alpha)(2-q^2)}\right]\lambda_{\alpha,1} \le \lambda_{\alpha,q} \le \lambda_{\alpha,1}$$
	for every $(\alpha,q)\in[0,1)\times(0,1)$. Therefore, $\lambda_{\alpha,q}\to \lambda_{\alpha,1}$ as $q\to 1$, and there is no discontinuity as the contractive character is lost. \\
	The case $\xi\in(0,1)$ is more involved. Lower values of $\xi$ make the constant $Q$ smaller, but may also restrict the possible values for $\alpha$ and $\lambda$, in view of inequality \eqref{E:H:2}. In the fully general case, if $\alpha$, $\lambda$ and $q$ satisfy
	$$\left[\frac{\alpha\lambda(1+\alpha)}{(1-\alpha)(1-\lambda)}\right]\left[\frac{1-\lambda+\lambda q^2}{1-\lambda+\lambda q^2-\alpha}\right] <1,$$
	then, there is $\xi_{\alpha,\lambda,q}\in(0,1)$ such that \eqref{E:H:2} holds for all $\xi\in[\xi_{\alpha,\lambda,q},1]$. As $q\to 1$, we recover \eqref{E:feasibility} as a limit case.
\end{enumerate}
	\section{Examples}\label{section:examples}

\subsection{Averaged Operators}\label{sec:averaged}

An operator $T:\HH\to \HH$ is $\gamma$-averaged if there is a nonexpansive operator $R:\HH\to \HH$ such that $T=(1-\gamma)I+\gamma R$. In this case, $\Fix(T)=\Fix(R)$.\\

Let $R:\HH\to \HH$ be nonexpansive and let $(\gamma_k)$ be a sequence in $(0,1)$. Setting $T_k=(1-\gamma_k)I+\gamma_kR$, \eqref{E:Algorithm} can be rewritten as
\begin{equation} \label{E:Algorithm_gamma}
	\left\{
	\begin{array}{rcl}
		y_k & = & x_k+\alpha_k(x_k-x_{k-1})\\
		x_{k+1} & = & (1-\gamma_k\lambda_k)y_k+\gamma_k\lambda_kR(y_k),
	\end{array}
	\right.
\end{equation} 
and Hypothesis \ref{H:2} becomes
$$\limsup_{k\to\infty}\big[\alpha_k(1+\alpha_k)+\big((\gamma_k\lambda_k)^{-1}-1\big)\alpha_k(1-\alpha_k)-\big((\gamma_{k-1}\lambda_{k-1})^{-1}-1\big)(1-\alpha_{k-1})\big]<0.$$
If $\gamma_k\lambda_k\to\eta>0$, this is
\begin{equation} \label{E:H2_constant_gamma}
	\eta(1-\alpha+2\alpha^2)<(1-\alpha)^2.
\end{equation}
It is not necessary to implement the algorithm using the operator $R$ explicitly. However, the interval for the relaxation parameters is enlarged, and it may be convenient to over-relax. We shall come back to this point in the numerical illustrations.\\

\subsection{Euler Iterations and Gradient Descent}  \label{S:Gradient}

An operator $B$ is {\it $\beta$-cocoercive} with $\beta>0$ if $\inner{Bx-By}{x-y}\geq \beta \norm{Bx - By}^2$ for all $x,y \in \HH$. \\

Let $B:\HH\to \HH$ be cocoercive with constant $\beta$, and let $(\rho_k)$ be a sequence in $(0,2\beta)$. For each $k\ge 1$, set
$$T_k=I-\rho_kB.$$
Then, $T_k$ is nonexpansive (thus quasi-nonexpansive) and $(\rho_k/2\beta)$-averaged. If $\rho_-:=\inf_{k\ge 1}\rho_k>0$, the family $(I-T_k)$ is asymptotically demiclosed. 
If $\lambda_k\rho_k\to\sigma$, Hypothesis \ref{H:2} becomes
$$\sigma(1-\alpha+2\alpha^2)<2\beta(1-\alpha)^2.$$
Now, let $f:\HH\to \HH$ be convex and differentiable, and assume $\nabla f$ is Lipschitz-continuous with constant $L$. Then, $B=\nabla f$ is cocoercive with constant $\beta=1/L$. If, moreover, $f$ is strongly convex with parameter $\mu$ and $\rho_k\le 2/(L+\mu)$, then $T_k$ is $q_k$-quasi-contractive with 
$$q_k=1-\frac{2\mu L\rho_k}{L+\mu}\le 1-\frac{2\mu L\rho_-}{L+\mu}=:q.$$
Therefore, $(T_k)$ is $q$-quasi-contractive. Considering the non-inertial case ($\alpha_k \equiv 0$), $\lambda_k \equiv 1$ and the fixed-sted choice $\rho_k=2/(\mu + L)$, the algorithm exhibits a rate of convergence
$$f(x_k) - f^* \leq \dfrac{L}{2}\left( \dfrac{Q-1}{Q+1}\right)^{2k}\norm{x_0 - x^*}^2,$$
where $Q=L/\mu$ is the {\it condition number}  (\cite[Theorem 2.1.15]{nesterov2018lectures}. Introducing the inertial term, and using  
$$\rho_k=1/L\qquad\hbox{and}\qquad \alpha_k\equiv\left(\frac{\sqrt{L}-\sqrt{\mu}}{\sqrt{L}+\sqrt{\mu}}\right),$$ 
it turns into  {\it constant step scheme, III} \cite{nesterov2018lectures}, which has a rate of convergence of 
$$f(x_k) - f^* \leq \min \left\lbrace \left( 1 - \sqrt{\dfrac{\mu}{L}} \right)^k, \dfrac{4L}{(2\sqrt{L}+k\sqrt{\mu})^2} \right\rbrace\left(f(x_0) - f^* + \frac{\mu}{2}\norm{x_0-x^*}^2\right).$$
Here, Hypothesis \ref{H:2} can be written as
$$\lambda<\dfrac{2Q}{1-\sqrt{Q}+2Q},$$
which gives the condition for the convergence of Nesterov's constant step scheme with constant relaxation $\lambda$.  

\subsection{Proximal and Forward-Backward Methods}\label{s:FB}

Let $M:\HH\to 2^{\HH}$ be maximally monotone and let $(\rho_k)$ be a positive sequence. The {\it proximal method} consists in iterating
\begin{equation} \label{E:T_PROX}
	z_{k+1}=(I+\rho_k M)^{-1}z_k,
\end{equation} 
for $k\ge 1$. The operator $T_k=J_{\rho_k M}:=(I+\rho_k M)^{-1}$ is nonexpansive, $\frac{1}{2}$-averaged, and $Z=\bigcap_{k\ge 1}\Fix(T_k)=M^{-1}0$. If $\lambda_k\to\lambda$, Hypothesis \ref{H:1} is reduced to
$$	\lambda(1-\alpha+2\alpha^2)<2(1-\alpha)^2.$$
As before, the family $(I-T_k)$ is asymptotically demiclosed at $0$ if $\inf_{k\geq 1}\rho_k>0$. To see this, let $(z_k)$ be a sequence in $\HH$ such that $z_k\wto z$ and $z_k-T_kz_k\to 0$. We must show that $0\in Mz$. By the definition of $T_k$, we have
$$\frac{1}{\rho_k}(z_k-T_kz_k)\in M(T_kz_k).$$
The left-hand side converges strongly to zero, while $T_kz_k\wto z$. We conclude by the weak-strong closedness of the graph of $M$. \\

Let $A:\HH\to 2^{\HH}$ be maximally monotone, let $B:\HH\to \HH$ be cocoercive with parameter $\beta$, and let $(\rho_k)$ be a sequence in $(0,2\beta)$. For each $k\ge 1$, set
$$T_k=(I+\rho_k A)^{-1}(I-\rho_kB).$$
Then, $T_k$ is $\gamma_k$-averaged with $\gamma_k=2\beta(4\beta-\rho_k)^{-1}$. If $\rho_k\to\rho$ and $\lambda_k\to\lambda$, then Hypothesis \ref{H:2} is equivalent to
$$	\lambda(1-\alpha+2\alpha^2)<\left(2-\frac{\rho}{2\beta}\right)(1-\alpha)^2.$$
As in the proximal case, the family $(I-T_k)$ is asymptotically demiclosed at $0$ if $\inf_{k\ge 1}\rho_k>0$. 

\subsection{Douglas-Rachford and primal-dual splitting}\label{S:DR}

Let $A, B:\HH\to 2^{\HH}$ be maximally monotone, and let $(r_k)$ be a positive sequence. The {\it Douglas-Rachford} splitting method consists in iterating $z_{k+1}=T_{r_k}z_k$, for $k\geq 1$, where
\begin{equation} \label{E:DRS}
	T_{r}=J_{r A}\circ\big(2J_{r B}-I\big)+\big(I-J_{rB}\big)=\frac{1}{2}\big(I+(2J_{r A}-I)\circ (2J_{r B}-I)\big).
\end{equation} 
The second expression shows that $T_r$ is averaged. Using the weak-strong closedness of the graphs of $A$ and $B$, and a little algebra, one proves that the family $\big(I-T_{r_k}\big)$ is asymptotically demiclosed if $\inf_{k\ge 0} r_k>0$. Finally, observe that $\Zer(A+B)=J_{rB}\Fix(T_r)$. \\

More generally, let $X$ and $Y$ be Hilbert spaces, and consider the {\it primal problem}, which is to find $\hat x\in X$ such that
$$0\in A\hat x+L^*BL\hat x,$$
where $A:X\to 2^X$ and $B:Y\to 2^Y$ are maximally monotone operators, and $L:X\to Y$ is linear and bounded. The {\it dual problem} is to find $\hat y\in Y$ such that
$$0\in B^{-1}\hat y-LA^{-1}(-L^*\hat y).$$
The primal and dual solutions, namely $\hat x$ and $\hat y$, are linked by the inclusions
$$-L^*\hat y\in A\hat x\qquad\hbox{and}\qquad L\hat x\in B^{-1}\hat y.$$

\begin{remark}
	Let $f:X\to \Rcupinf$ and $g:Y\to\Rcupinf$ be closed and convex, and set $A=\partial f$ and $B=\partial g$. The inclusions above are the optimality conditions for the primal and dual (in the sense of Fenchel-Rockafellar) optimization problems
	\begin{equation}\label{eq:min_problem}
		\min_{x\in X}\{f(x)+g(Lx)\}\qquad\hbox{and}\qquad \min_{y\in Y}\{g^*(y)+f^*(-L^*y)\},   
	\end{equation}
	respectively. Douglas-Rachford splitting applied to $ A=\partial g^*$ and $ B=\partial\big(f^*\circ(-L^*)\big)$ yields the {\it alternating direction method of multipliers} (see \cite{gabay1976dual}).
\end{remark}

In order to find a primal-dual pair, the {\it primal-dual} splitting algorithm (see \cite{chambolle2011first}) iterates:
\begin{equation}\label{eq:splitting_iterations}
	\left\{\begin{array}{rcl}
		x_{k+1} & = & J_{\tau A}\big(x_k-\tau L^*y_k\big) \\
		y_{k+1} & = & J_{\sigma B^{-1}}\big(y_k+\sigma L(2x_{k+1}-x_k)\big), 
	\end{array}\right.    
\end{equation}
with $\tau\sigma\|L\|^2\le 1$. The algorithm can be expressed as 
$(x_{k+1},y_{k+1})=T(x_k,y_k)$, where $T: X\times Y \to X\times Y$ is a $1/2$-averaged operator (see \cite[Remark 4.34]{bauschke2011convex}). \\
	
	An inertial version of the primal-dual iterations is given by
	\begin{equation}
		\left\{\begin{array}{l}
			(y_k,v_k)=(x_k,u_k)+\alpha_k\parenc{(x_k,u_k)-(x_{k-1},u_{k-1})}   \\
			p_{k+1}=J_{\tau A}(y_k-\tau L^* v_k)\\
			q_{k+1}=J_{\sigma B^{-1}}(v_k+\sigma L(2p_{k+1}-y_k))\\
			(x_{k+1},u_{k+1})= (1-\lambda_k)(y_k,v_k)+\lambda_k(p_{k+1},q_{k+1}),
		\end{array}\right.
		\label{seq:inertial_splitting}
	\end{equation}
	with appropriate sequences $\alpha_k$ and $\lambda_k$.\\
	
	In \cite{briceno2019primal}, the authors propose the {\it Split Douglas-Rachford} algorithm 
	\begin{equation} \label{E:SDRA}
		\left\{\begin{array}{rcl}
			v_k & = & \Sigma\big(I-J_{\Sigma^{-1}B}\big)\big(Lx_k+\Sigma^{-1}y_k\big) \\
			x_{k+1} & = & J_{\Upsilon A}\big(x_k-\Upsilon L^*v_k\big) \\
			y_{k+1} & = & \Sigma L(x_{k+1}-x_k)+v_k, 
		\end{array}\right.
	\end{equation} 
	where $\Upsilon$ and $\Sigma$ are elliptic linear operators that induce an {\it ad-hoc} metric and account for preconditioning.

	
	\subsection{Three Operator Splitting}\label{section:3op}
	
	Given three maximally monotone operators $A, B, C$ defined on the Hilbert space $H$, we wish to find $\hat x \in H$ such that 
	\begin{equation}\label{prob:3op}
		0 \in A\hat x + B\hat x + C\hat x.
	\end{equation}
	If $C$ is $\beta$-cocoercive, the {\it three-operator} splitting method \cite{davis2017three} generates a sequence $(z_k)$ by 
	\begin{equation}
		\left\{\begin{array}{l}
			x^B_k=J_{\rho B}(z_k)  \\
			x^A_k=J_{\rho A}(2x^B_k-z_k-\rho Cx_k^B)\\
			z_{k+1}= z_k+\lambda_k(x^A_k-x^B_k)
		\end{array}\right.
		\label{seq:threeOperator}
	\end{equation}
	starting from a point $z_0\in H$. Here $\rho \in(0,2\beta)$,  $\lambda_k\in(0,1/\gamma)$ and 
	\begin{equation}
		\gamma=\frac{2\beta}{4\beta-\rho}.    
		\label{eq:gamma}
	\end{equation}
	This recurrence is generated by iterating the $\gamma$-averaged operator
	$$T= I-J_{\rho B}+J_{\rho A}\circ(2J_{\rho B}-I-\rho C\circ J_{\rho B}),$$
	and we have $\Zer(A+B+C) = J_{\rho B} (\Fix T)$. Also, it gives the forward-backward method if $B=0$ and the Douglas-Rachford method if $C=0$. An inertial version is given by
	\begin{equation}
		\left\{\begin{array}{l}
			u_k=z_k+\alpha_k(z_k-z_{k-1})\\
			x^B_k=J_{\rho B}(u_k)  \\
			x^A_k=J_{\rho A}(2x^B_k-u_k-\rho Cx_k^B)\\
			z_{k+1}= u_k+\lambda_k(x^A_k-x^B_k),
		\end{array}\right.
		\label{seq:inertialThreeOperator}
	\end{equation}
	for appropriate choices of $\alpha_k$, $\lambda_k$. One particular instance is given by the optimization problem
	\begin{equation}\label{prob:3opconvex}
		\min f(x) + g(x) + h(Lx),
	\end{equation}
	where $f,g,h$ are closed and convex, $h$ has a $(1/\beta)$-Lipschitz-continuous gradient, and $L$ is a bounded linear mapping.   
	\section{Numerical Illustrations}\label{section:numerical}

In this section, we test the performance of the algorithm given by iterations \eqref{E:Algorithm} in two of the settings described in Section \ref{section:examples}. More precisely, we apply an inertial primal-dual splitting method to solve a TV-based denoising problem, and an inertial three-operator splitting algorithm to in-paint a corrupted image.

\subsection{Primal-Dual Splitting and TV-based Denoising} 

The algorithm will be tested in an image processing framework. Consider the problem

\begin{equation}\label{eq: TVproblem}
	\min_{x \in \R^{N_1 \times N_2}} F^{TV}(x) := \frac{1}{2} \normc{Rx-b} + w\norm{\nabla x}_1, 
\end{equation}
where $x \in \R^{N_1 \times N_2}$ is an image to recover from  a noisy observation $b\in \R^{M_1\times M_2}$, $R:\R^{N_1\times N_2}\to \R^{M_1\times M_2}$ is a blur operator, $w$ is a positive parameter, and $\nabla: x \mapsto \nabla x = (D_1x,D_2x)$ is the classical discrete gradient, whose adjoint $\nabla^*$ is the discrete divergence. A formulation for the gradient and divergence operators can be seen on \cite{chambolle2010introduction}. In these experiments, $R$ will be a Gaussian blur of size $9\times 9$, standard deviation $4$ and relative boundary conditions (see \cite{hansen2006deblurring} for details on the construction of the operator), and $w=10^{-4}$. Considering the original image $\bar{x}$ in Figure \ref{fig:image_TV1} composed by $256\times 256$ pixels, the observation $b$ is generated as $b=R\bar{x}+e$, where $e$ is an additive zero-mean white Gaussian noise with standard deviation $10^{-3}$ (Figure \ref{fig:image_TV2}).\\

Setting $f=0$, $g:(u,v^1,v^2) \mapsto \frac{1}{2} \normc{u-b} + w \norm{v^1}_1 + w \norm{v^2}_1$ and $L:x \mapsto (Rx,D_1 x , D_2 x)$, the problem \eqref{eq: TVproblem} can be formulated as \eqref{eq:min_problem}, and solved via \eqref{seq:inertial_splitting}. Since
\begin{equation}
	\prox_{\sigma g^*}: (u,v^1,v^2) \mapsto \left( \dfrac{u-\sigma b}{\sigma + 1}, v^1 - \sigma \prox_{\frac{w}{\sigma}\norm{\cdot}_1}\left(\dfrac{v^1}{\sigma} \right), v^2 - \sigma \prox_{\frac{w}{\sigma}\norm{\cdot}_1}\left(\dfrac{v^2}{\sigma} \right) \right)
\end{equation}
we are lead to Algorithm \ref{algorithm:TV}. 

\begin{algorithm}[h]
	\footnotesize
	\SetAlgoLined
	
	Choose $x_0,x_1\in \R^{N_1 \times N_2}$, $u_0,u_1\in \R^{m_1 \times m_2}$, $v_0^1, v_1^1, v_0^2, v_1^2 \in \R^{N_1 \times N_2}$, $(\lambda_k)_{k \in \N}$ and $(\alpha_k)_{k \in \N}$ such that hypotheses of Theorem \ref{T:1} are fulfilled, $\tau$ and $\sigma$ such that $\tau\sigma\norm{L}^2 \leq 1$, $\varepsilon>0$ and $r_0>\varepsilon$ \; 
	\While{$r_k>\varepsilon$ }{
		$(\bar{x}_k,\bar{u}_k,\bar{v}^1_k,\bar{v}^2_k)=(x_k,u_k,v_k^1,v_k^2)+\alpha_k [(x_k,u_k,v_k^1,v_k^2)-(x_{k-1},u_{k-1}, v_{k-1}^1,v_{k-1}^2 )]$\;
		$p_{k+1}=\bar{x}_k - \tau R^*\bar{u}_k - \tau D_1^* \bar{v}^1_k - \tau D_2^* \bar{v}^2_k$\;
		$q_{k+1}=(\bar{u}_k + \sigma R(2p_{k+1} - \bar{x}_k) - \sigma b) / (\sigma + 1)$\;
		$w_{k+1}^1= \bar{v}^1_k + \sigma D_1(2p_{k+1} - \bar{x}_k) - \sigma \prox_{w \norm{\cdot}_1/\sigma} (\bar{v}^1_k / \sigma + D_1 (2p_{k+1}-\bar{x}_k))$\;
		$w_{k+1}^2= \bar{v}^2_k + \sigma D_2(2p_{k+1} - \bar{x}_k) - \sigma \prox_{w \norm{\cdot}_1/\sigma} (\bar{v}_k^2 / \sigma + D_2 (2p_{k+1}-\bar{x}_k))   $\;
		$(x_{k+1},u_{k+1},v_{k+1}^1, v_{k+1}^2) = (1 - \lambda_k)(\bar{x}_k,\bar{u}_k,\bar{v}_k^1,\bar{v}_k^2) + \lambda_k (p_{k+1},q_{k+1},w_{k+1}^1, w_{k+1}^2) $ \;
		$r_k = \mathcal{R}((x_{k+1},u_{k+1},v_{k+1}^1, v_{k+1}^2),(x_k,u_k,v_k^1,v_k^2))$ 
	}\
	\Return $(x_{k+1},u_{k+1} ,v_{k+1}^1, v_{k+1}^2 )$
	\caption{}
	\label{algorithm:TV}
\end{algorithm}

For a stopping criterion, we consider the relative error
\begin{equation}\label{eq:stop_iter}
	\mathcal{R}(x_{k+1},x_k) \mapsto \dfrac{\norm{x_{k+1} - x_k}}{\norm{x_k}}.
\end{equation}
Since the involved operator is $1/2$-averaged (see \cite{briceno2015forward}), we may set $\lambda_k \equiv \lambda \in (0,2)$, as explained in Section \ref{sec:averaged}.\\ 

The algorithm is tested for 17 combinations of $\tau,\sigma$ satisfying the critical condition $\tau \sigma \norm{L}^2 = 1 $ (according to \cite{briceno2019primal}, this tends to yield the best performance). The number $\norm{L}$ is computed using an adaptation of \cite[Algorithm 12]{pustelnik2010methodes}. \\

{\bf Comparison in terms of the parameters $\tau$ and $\sigma$.} In a first stage, we compare the performance of the primal-dual splitting algorithm given by \eqref{eq:splitting_iterations} (that is, Algorithm \ref{algorithm:TV} with $\alpha_k \equiv 0$), and its inertial counterpart \eqref{seq:inertial_splitting}, with $\lambda_k \equiv 1$. 
The sequence $(\alpha_k)_{k \in \N}$ is
\begin{equation}
	\alpha_k=\alpha\left(1-\frac{1}{k^2}\right),
	\label{eq:alphasequence}
\end{equation}
with $\alpha = 1/(3+0.0001)$ (condition \eqref{E:H2_constant_gamma} with $\eta=\lambda/2$ gives the constraint $\alpha < 1/3$). Table \ref{tab:table_TV1} shows the execution time, number of iterations, and the value for the objective value reached, using a tolerance $\varepsilon=10^{-5}$. These results are depicted graphically, along with the percentage of reduction, in Figure \ref{fig:tv_tol5}. The recovered images are collected in Figures \ref{fig:image_TV3} and \ref{fig:image_TV4}. \\

\begin{table}[htbp]
	\footnotesize
	\centering
	\begin{tabular}{rrrrrrrrr}
		\toprule
		& & &\multicolumn{3}{c}{Original algorithm} & \multicolumn{3}{c}{Inertial algorithm} \\
		Case &      $\tau$ &    $\sigma$ &  Time &  Iterations & $F^{TV}(x)$  &  Time &  Iterations &  $F^{TV}(x)$  \\
		\midrule
		1 & 0.0004 & 282.8427 & 72.59 & 1565 & 7.30 & 55.11 & 1095 & 7.13 \\
		2 & 0.0010 & 122.6475 & 115.66 & 2437 & 2.84 & 86.97 & 1741 & 2.66 \\
		3 & 0.0024 & 53.183 & 110.16 & 2330 & 1.35 & 83.98 & 1672 & 1.27 \\
		4 & 0.0054 & 23.0614 & 98.28 & 2077 & 0.7566 & 72.33 & 1446 & 0.7341 \\
		5 & 0.0125 & 10 & 94.80 & 2015 & 0.4624 & 69.59 & 1394 & 0.4537 \\
		6 & 0.0288 & 4.3362 & 105.19 & 2253 & 0.2975 & 77.83 & 1562 & 0.2928 \\
		7 & 0.0665 & 1.8803 & 122.23 & 2593 & 0.2107 & 89.83 & 1773 & 0.2091 \\
		8 & 0.1533 & 0.8153 & 156.34 & 3248 & 0.1592 & 112.09 & 2184 & 0.1589 \\
		9 & 0.3536 & 0.3536 & 140.91 & 2922 & 0.1428 & 101.69 & 1956 & 0.1427 \\
		10 & 0.8153 & 0.1533 & 139.50 & 2856 & 0.1350 & 98.97 & 1908 & 0.1350 \\
		11 & 1.8803 & 0.0665 & 151.08 & 3123 & 0.1312 & 107.72 & 2084 & 0.1312 \\
		12 & 4.3362 & 0.0288 & 108.08 & 2249 & 0.1303 & 78.03 & 1503 & 0.1303 \\
		13 & 10 & 0.0125 & 60.28 & 1238 & 0.1301 & 42.78 & 833 & 0.1301 \\
		14 & 23.0614 & 0.0054 & 47.61 & 983 & 0.1302 & 35.70 & 693 & 0.1302 \\
		15 & 53.1830 & 0.0024 & 70.78 & 1466 & 0.1302 & 54.61 & 1065 & 0.1302 \\
		16 & 122.6475 & 0.0010 & 119.22 & 2471 & 0.1302 & 89.91 & 1762 & 0.1302 \\
		17 & 282.8427 & 0.0004& 179.22 & 3767 & 0.1302 & 150.52 & 2999 & 0.1302 \\
		\bottomrule
	\end{tabular}
	\caption{Execution time, number of iterations and final function value for the original primal-dual algorithm and the inertial version, with tolerance $\varepsilon=10^{-5}$. }
	\label{tab:table_TV1}
\end{table}

\begin{figure}[htbp]
	\centering
	\subfloat{\includegraphics[width=0.32\textwidth]{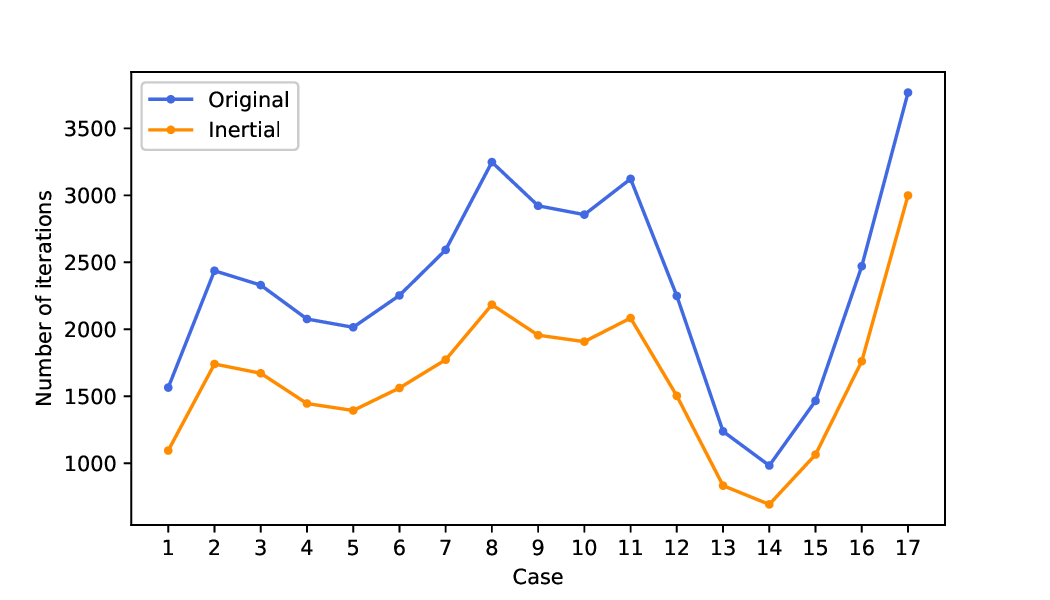}}
	\subfloat{\includegraphics[width=0.32\textwidth]{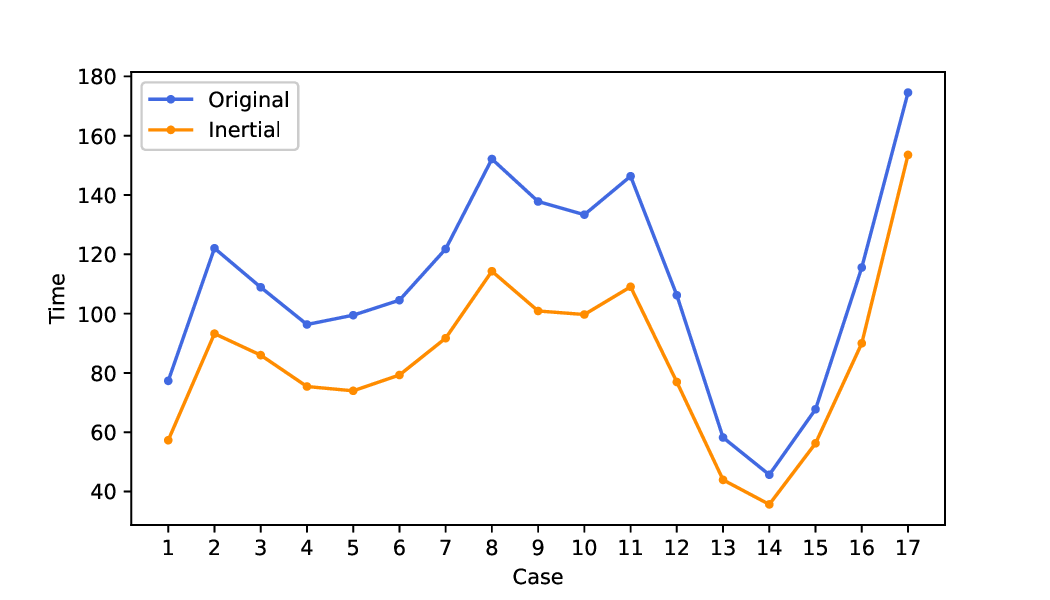}}
	\subfloat{\includegraphics[width=0.32\textwidth]{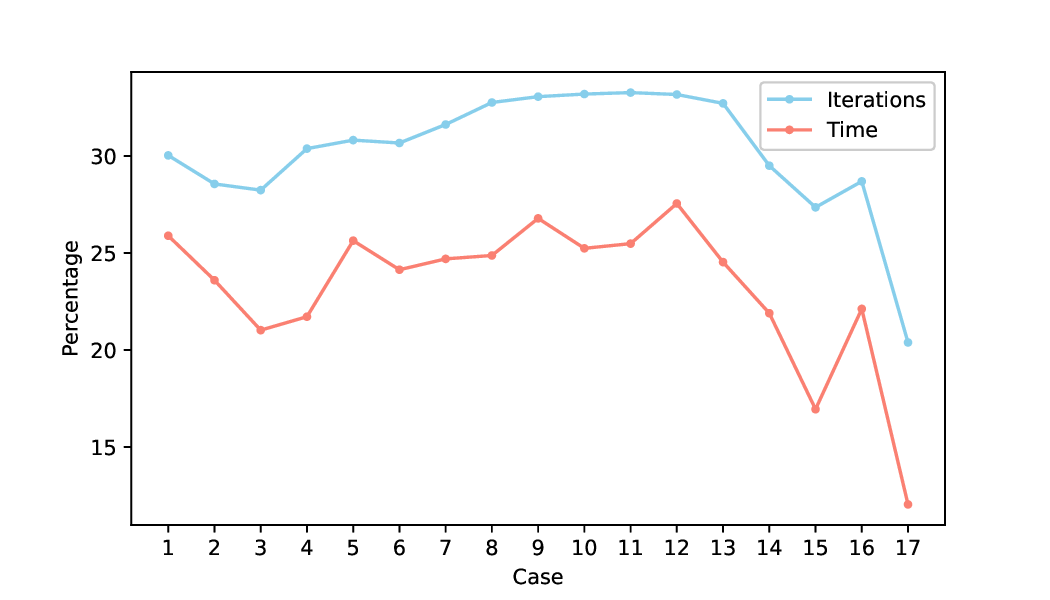}}
	\caption{Number of iterations (left), execution time (center), and percentage of reduction (right), from Table \ref{tab:table_TV1}.}
	\label{fig:tv_tol5}
\end{figure}


\begin{figure}[htbp]
	\centering
	\subfloat[Original Image]{\includegraphics[width=0.2\textwidth]{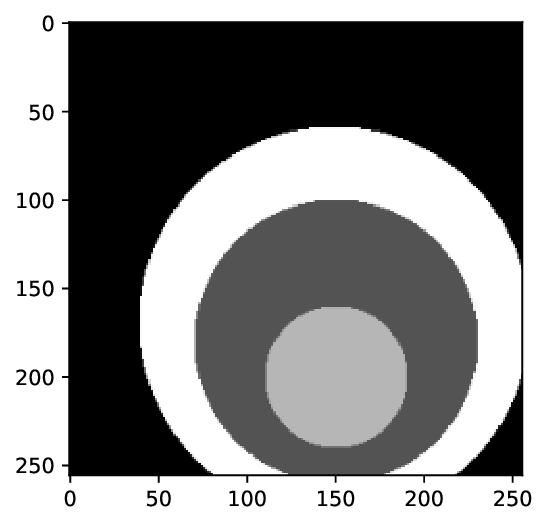}\label{fig:image_TV1} }
	\subfloat[Blurred Image]{\includegraphics[width=0.2\textwidth]{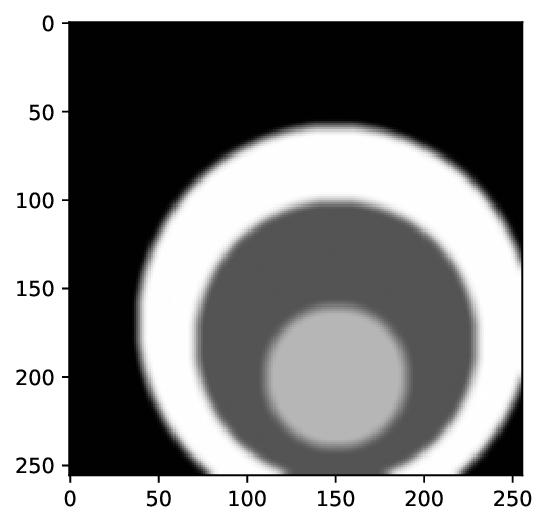}\label{fig:image_TV2} } 
	\subfloat[Recovered without inertia]{\includegraphics[width=0.2\textwidth]{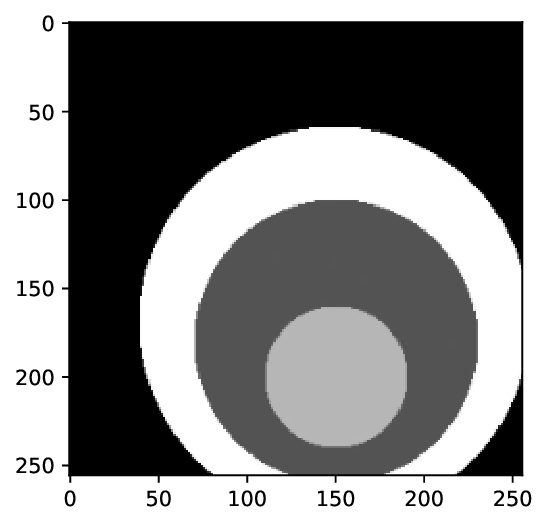}\label{fig:image_TV3} }
	\subfloat[Recovered with inertia]{\includegraphics[width=0.2\textwidth]{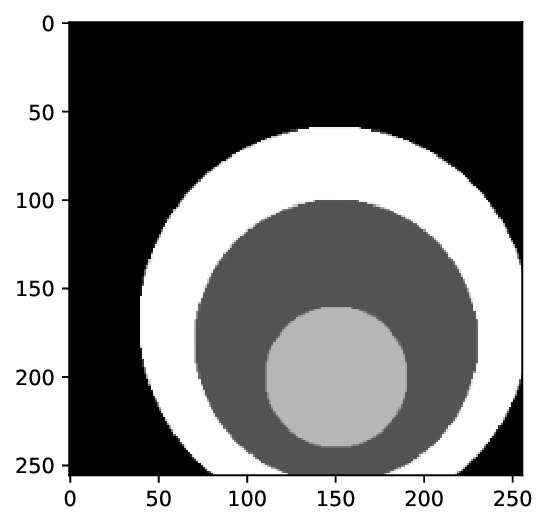}\label{fig:image_TV4} }
	\caption{Original, blurred and recovered images. Lowest recovered value $F^{TV}(x)=0.1301$ (case 13, both methods).}
\end{figure}

{\bf Comparison in terms of the relaxation parameter $\lambda$.} For both algorithms, case 14 showed the best performance in terms of iterations and execution time. We now assess the performance of the inertial algorithm with different values for $\lambda_k \equiv \lambda \in (0,2)$, and the corresponding inertial parameters fulfilling condition \eqref{E:H2_constant_gamma}. The results are shown in Table \ref{tab:table_TV3}, along with the value of $\alpha$ used in \eqref{eq:alphasequence}. A graphic depiction is shown as heatmaps in Figure \ref{fig:heat}. Larger values of the relaxation parameter $\lambda$ resulted in an improvement in the performance of both algorithms, but limit the impact of inertia, as it reduces the feasible range for the limit $\alpha$. A more thorough study on the selection of these parameters is the object of a forthcoming article.\\


\begin{table}[htbp]
	\footnotesize
	\centering
	\begin{tabular}{rrrrrrrrrr}
		\toprule
		& &\multicolumn{3}{c}{Original algorithm} & \multicolumn{3}{c}{Inertial algorithm} &  \% Iterations & \% Time \\
		$\lambda$ &  $\alpha$  & Time &  Iterations &  $F^{TV}(x)$ &  Time &  Iterations &  $F^{TV}(x)$ & reduction &  reduction\\
		\midrule
		0.2 & 0.6534 & 119.16 & 2592 & 0.1303 & 49.23 & 992 & 0.1304 & 61.73 & 58.69 \\
		0.4 & 0.5425 & 74.44 & 1589 & 0.1302 & 40.45 & 799 & 0.1303 & 49.72 & 45.66 \\
		0.6 & 0.4619 & 62.28 & 1341 & 0.1302 & 39.06 & 773 & 0.1302 & 42.36 & 37.28 \\
		0.8 & 0.3943 & 54.05 & 1146 & 0.1302 & 33.94 & 730 & 0.1302 & 36.30 & 37.21 \\
		1.0 & 0.3333 & 46.12 & 983 & 0.1302 & 34.47 & 693 & 0.1302 & 29.50 & 25.26 \\
		1.2 & 0.2748 & 41.16 & 861 & 0.1301 & 35.17 & 684 & 0.1302 & 20.56 & 14.55 \\
		1.4 & 0.1352 & 38.22 & 771 & 0.1301 & 34.45 & 675 & 0.1301 & 12.45 & 9.86 \\
		1.6 & 0.0967 & 33.89 & 718 & 0.1301 & 33.59 & 655 & 0.1301 & 8.77 & 0.89 \\
		1.8 & 0.0535 & 32.28 & 679 & 0.1301 & 32.62 & 657 & 0.1301 & 3.24 & -1.05 \\
		\bottomrule
	\end{tabular}
	\caption{Execution time, number of iterations, final function value and reduction percentage for the original primal-dual algorithm and
		the inertial version (case 14), with tolerance $\varepsilon=10^{-5}$.}
	\label{tab:table_TV3}
\end{table}

\begin{figure}[!htbp]
	\centering
	\includegraphics[width=0.45\textwidth,height=0.35\textwidth]{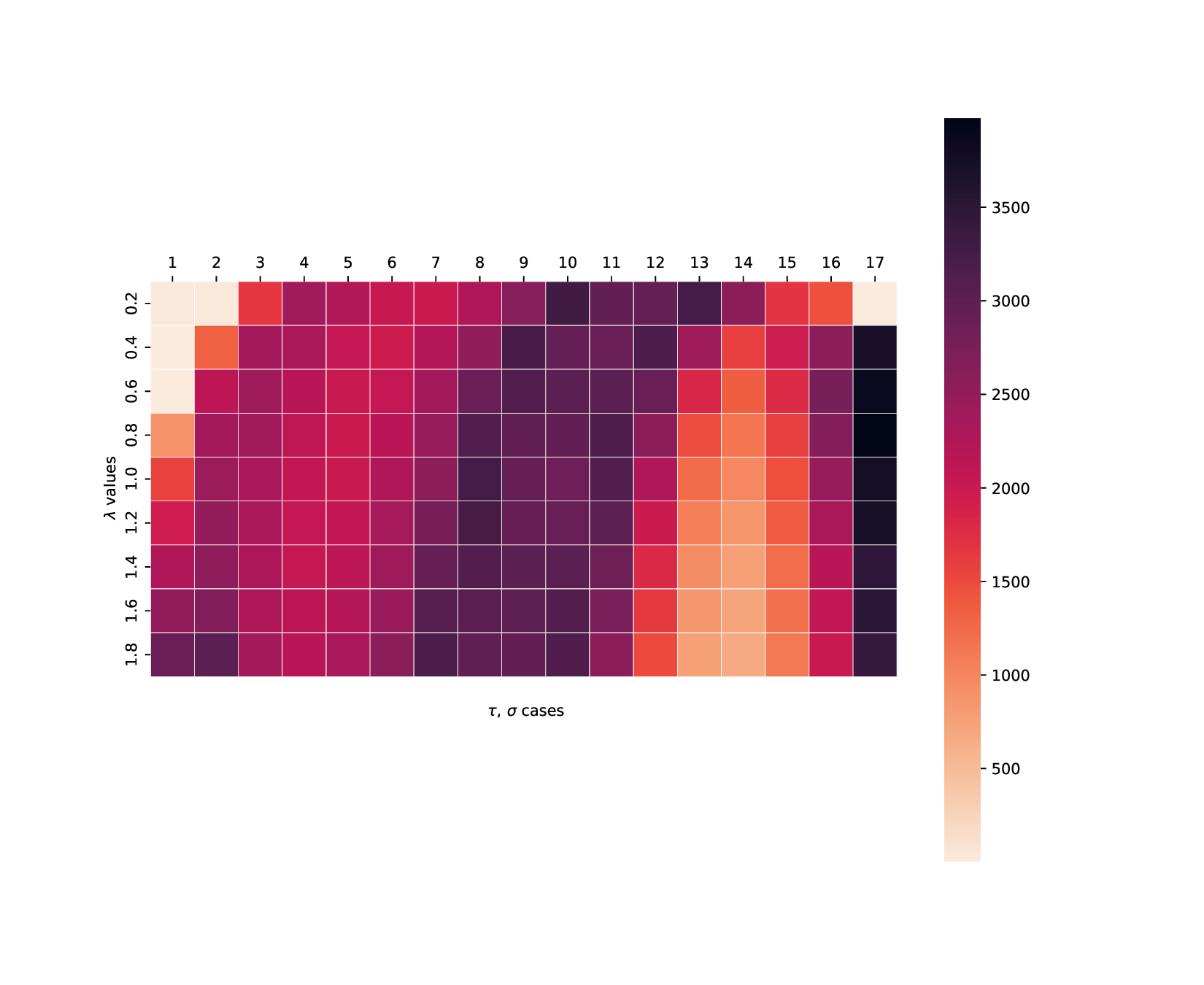}\label{fig:image_TV6}
	\includegraphics[width=0.45\textwidth,height=0.35\textwidth]{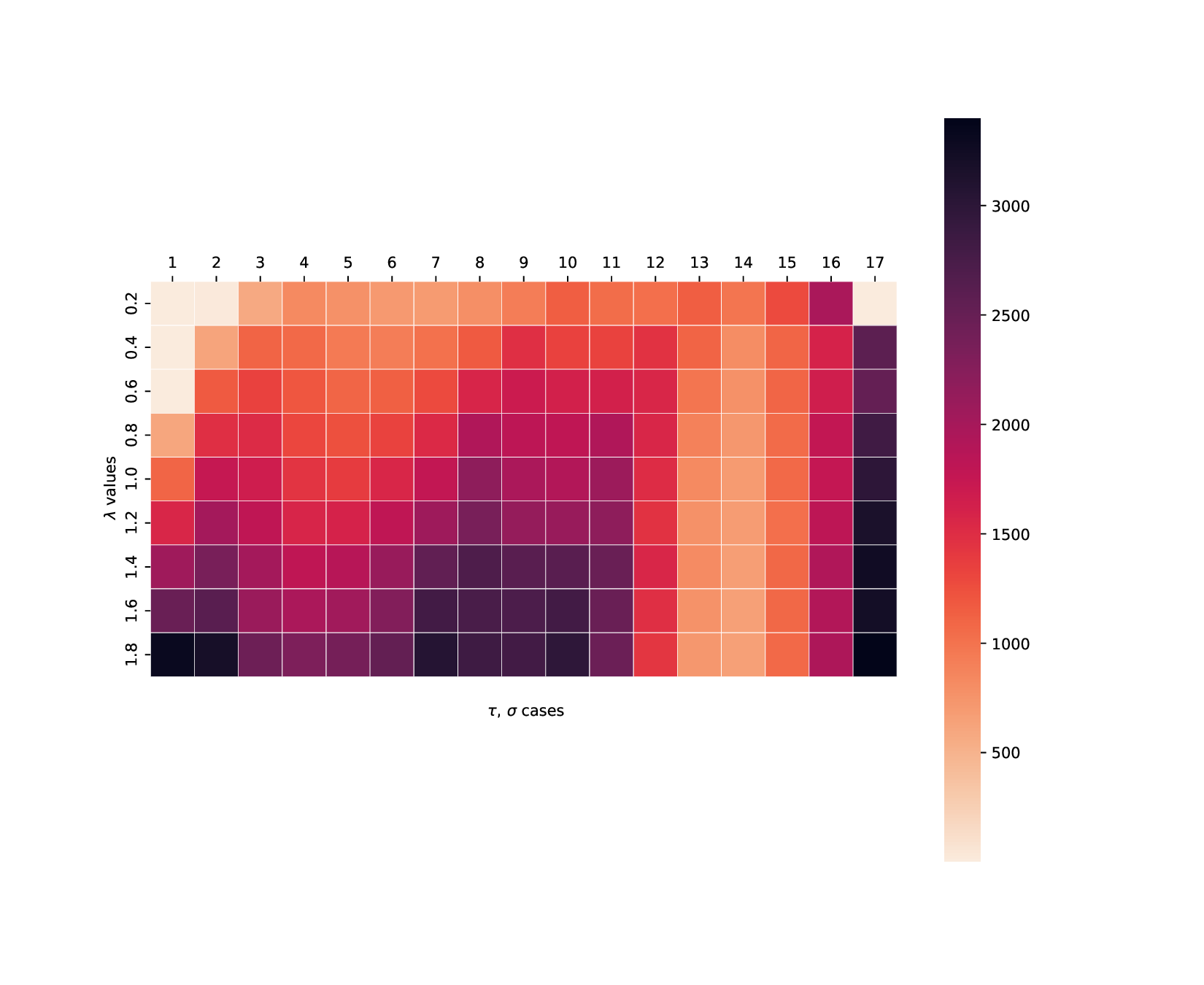}\label{fig:image_TV7}     
	\caption{Average number of iterations performed by the original (left) and inertial (right) algorithms, with tolerance $\varepsilon=10^{-5}$, for each value of $\lambda$, and each case of $\tau$ and $\sigma$, from Table \ref{tab:table_TV3}.}
	\label{fig:heat} 
\end{figure}


\if{ 
	Figure \ref{fig:heat} shows the iterations made by the original and the inertial algorithm for each case of $\lambda$, $\sigma$ and $\tau$. It can be noticed the general reduction in the amount of iterations made by the inertial algorithm, and for the case 14, for $\sigma$ and $\tau$, a stabilization can be noticed, for the different values of $\lambda$. The darkest points of the heat map, corresponds to cases where the algorithm reached the maximum of the iterations tolerated for the Algorithm, which is set to 4000.    
}\fi 

Finally, Figure \ref{fig:image_TV9} shows the evolution of the function values, the distance to the limit and the residuals, all in logarithmic scale, for case 14. The figure also includes the plot of $k\norm{z_k - Tz_k}^2$. Theorem \ref{T:1} states that the residuals show an non-asymptotic rate given by \eqref{E:nonasymptotic}, so we can conjecture an asymptotic rate of $o(1/k)$. 

\begin{figure}[!htbp]
	\centering
	\includegraphics[scale=0.35]{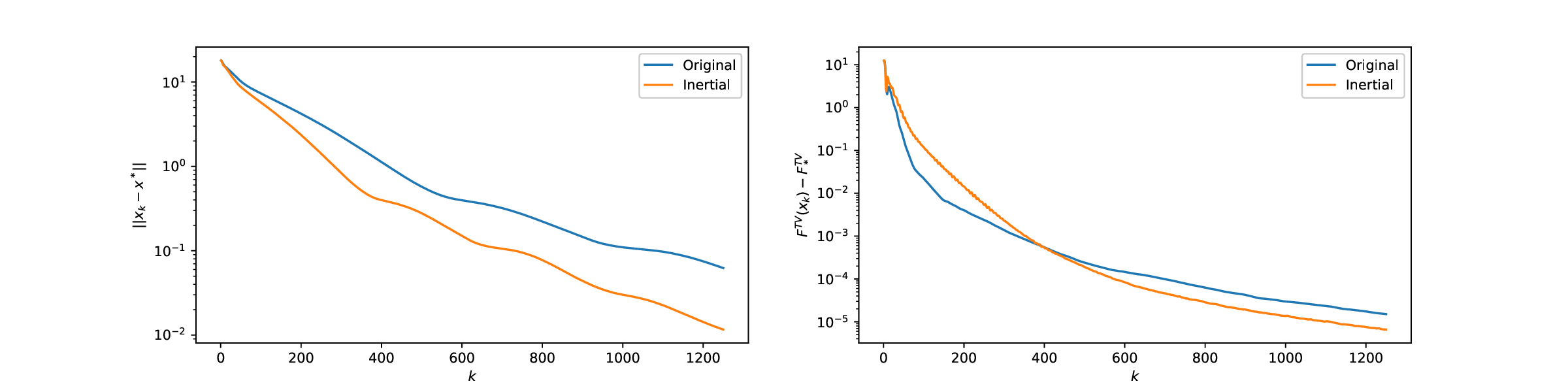}
	\includegraphics[scale=0.35]{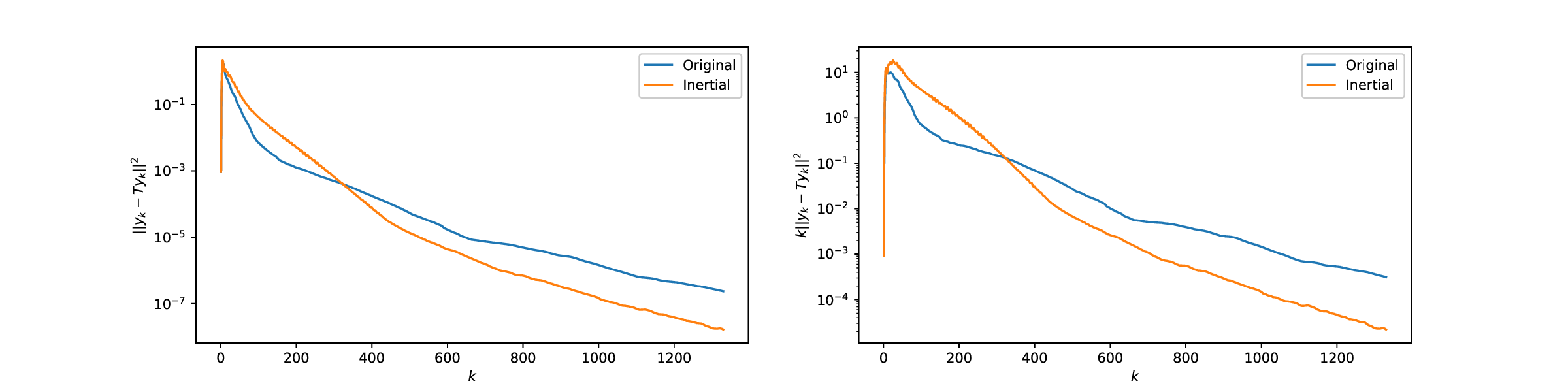}
	\caption{Evolution to the distance to the computed solution (top left), objective function values (top right), residuals $\norm{z_k - Tz_k}^2$ (bottom left) and $k\norm{z_k - Tz_k}^2$ (bottom right), for case 14. }
	\label{fig:image_TV9} 
\end{figure}


\subsection{Three-Operator Splitting and Image In-painting}

Suppose that $Z$ is a color image represented as a 3-D tensor where $Z(:,:,1), Z(:,:,2), Z(:,:,3)$ are the red, green and blue channels, respectively. Consider a damaged image $Y$, with randomly erased pixels, represented by the white color. The positions of the erased pixels are known. Denote $\mathcal{A}$ the linear operator that selects the set of correct entries of $Z$ (and so $\mathcal{A}^*$ is the {\it zero upsampling} operator). The objective is to recover the image, by filling the erased pixels. 
Following \cite{davis2017three} we consider the following formulation of the in-panting problem:
\begin{equation}
	\min_{Z\in \mathcal{H}} F(Z) := \frac{1}{2}\normc{\mathcal{A}(Z-Y)}+ w\norm{Z_{(1)}}_*+w\norm{Z_{(2)}}_*,
	\label{problem:imageInpanting}
\end{equation}
where $\mathcal{H}$ is the set of 3-D tensors, $Z_{(1)}$ is the matrix $[Z(:,:,1) \,Z(:,:,2) \,Z(:,:,3)]$, $Z_{(2)}$ is the matrix $[Z(:,:,1)^T \,Z(:,:,2)^T \,Z(:,:,3)^T]^T$, $\norm{\cdot}_*$ denotes the matrix nuclear norm and $w$ is a penalty parameter, which we take equal to 1 here, for simplicity. This problem fits in the context of \eqref{prob:3opconvex}, with $f(Z)=g(Z)=\norm{Z}_*$ and  $h(Z)=\frac{1}{2}\norm{Z-Y}_2^2$. In this case, the operator $\nabla (h \circ \mathcal{A})$ is cocoercive with constant 1. With the error function $\cali R$ defined in \eqref{eq:stop_iter}, the iterations defined by \eqref{seq:inertialThreeOperator} lead to Algorithm \ref{algorithm:threeoperators}. 

\begin{algorithm}[htbp]
	\footnotesize
	\SetAlgoLined
	Choose $Z_0,Z_1\in \R^{m\times n}$, $(\lambda_k)_{k \in \N}$ and $(\alpha_k)_{k \in \N}$ such that hypotheses of Theorem \ref{T:1} are fulfilled, $\rho \in (0,2)$, $\varepsilon>0$ and $r_0>\varepsilon$ \;
	\While{$r_k>\varepsilon$ }{
		$U_k=Z_k+\alpha_k(Z_k-Z_{k-1})$\;
		$X^g_k=\prox_{\rho g}(U_k)$\;
		$Z_{k+\frac{1}{2}}=2X_k^g-U_k-\rho \mathcal{A}^*\nabla h(\mathcal{A}X_k^g)$\;
		$Z_{k+1}=U_k+\lambda_k(\prox_{\rho f}(Z_{k+\frac{1}{2}})-X_k^g)$\;
		$r_{k+1}=\mathcal{R}(Z_{k+1},Z_{k})$
	}\
	Return $Z_{n+1}, X^g_n$\;
	\caption{}
	\label{algorithm:threeoperators}
\end{algorithm}

As in the previous section, Algorithm \ref{algorithm:threeoperators} will be tested in the case $\alpha_k \equiv 0$ (the algorithm studied in \cite{davis2017three}) and, for the inertial version, 
\begin{equation}
	\alpha_k= \paren{1-\frac{1}{k}}\alpha, 
	\label{seq:alphaInpainting}
\end{equation}
where $\alpha$ satisfies the condition \eqref{E:H2_constant_gamma}. The corresponding algorithms will be referred to as original and inertial, respectively. Algorithm \eqref{algorithm:threeoperators} returns both the value of $Z_{k}$ and $X_{k}^g$, since the latter represents the image solution of the problem. Throughout this section, the initial points are both set to zero. \\

{\bf Comparison in terms of the number of erased pixels.} Between 10000 and 250000 pixels are randomly erased from the image in Figure~\ref{fig:inpainting_original} to obtain the one in Figure \ref{fig:inpainting_erased}. We compare the number of iterations and execution time needed by both methods with step size $\rho=1$ and $\lambda_k\equiv 1$, for a tolerance of $10^{-3}$. The results are shown in Figure~\ref{fig:compare_erased}. The reduction stands between 12\% and 22\% in most cases, and the improvement seems to increase with the number of erased pixels.

\begin{figure}[htb]
	\centering
	\subfloat{\includegraphics[width=0.32\textwidth,height=0.2\textwidth]{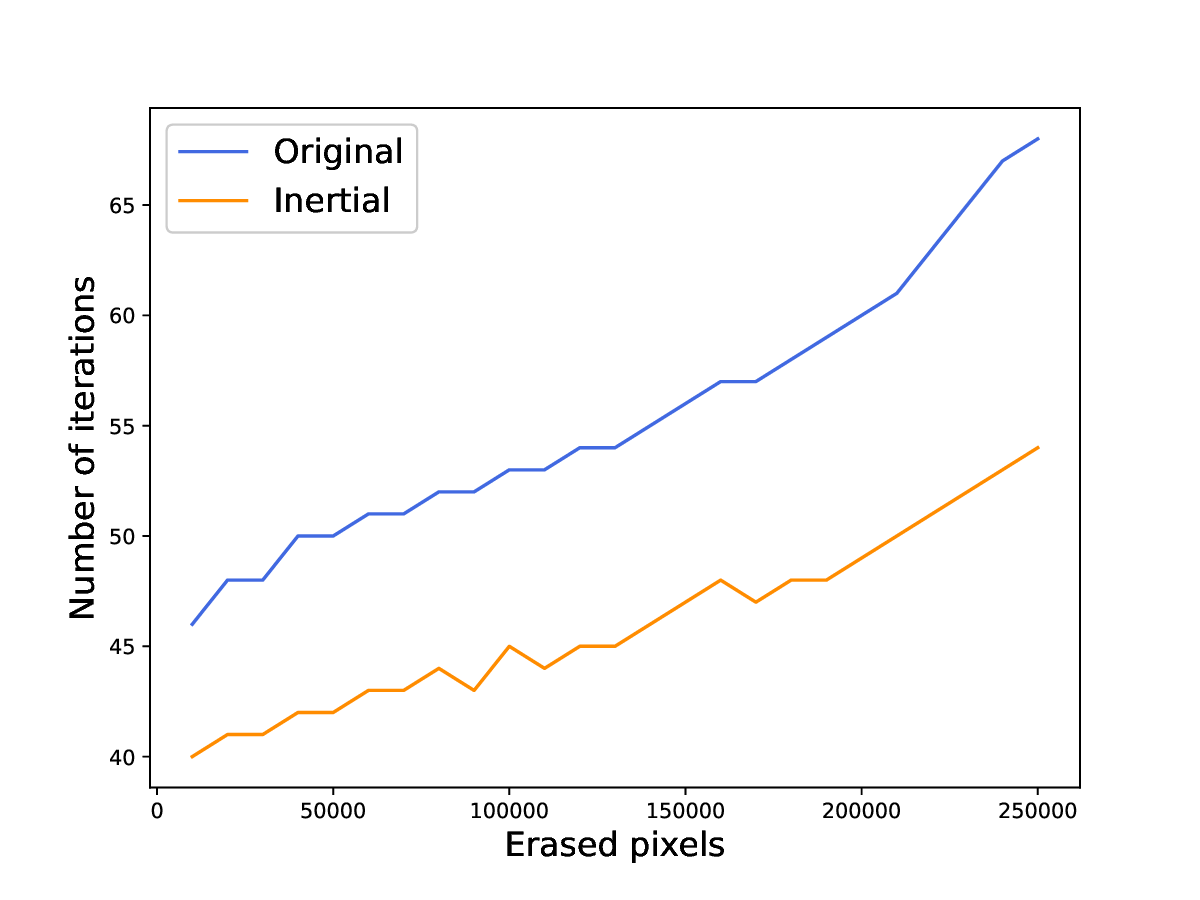}}
	\subfloat{\includegraphics[width=0.32\textwidth,height=0.2\textwidth]{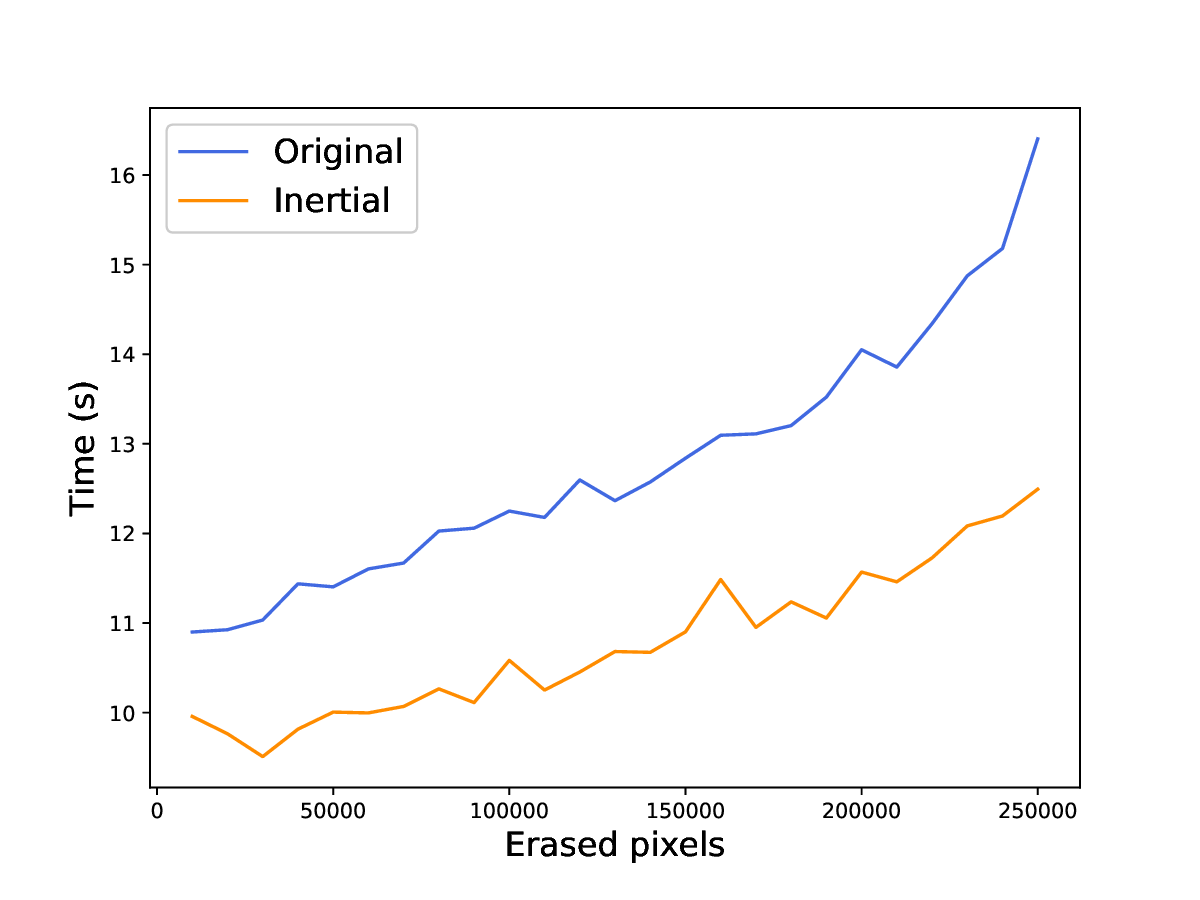}}
	\subfloat{\includegraphics[width=0.32\textwidth,height=0.2\textwidth]{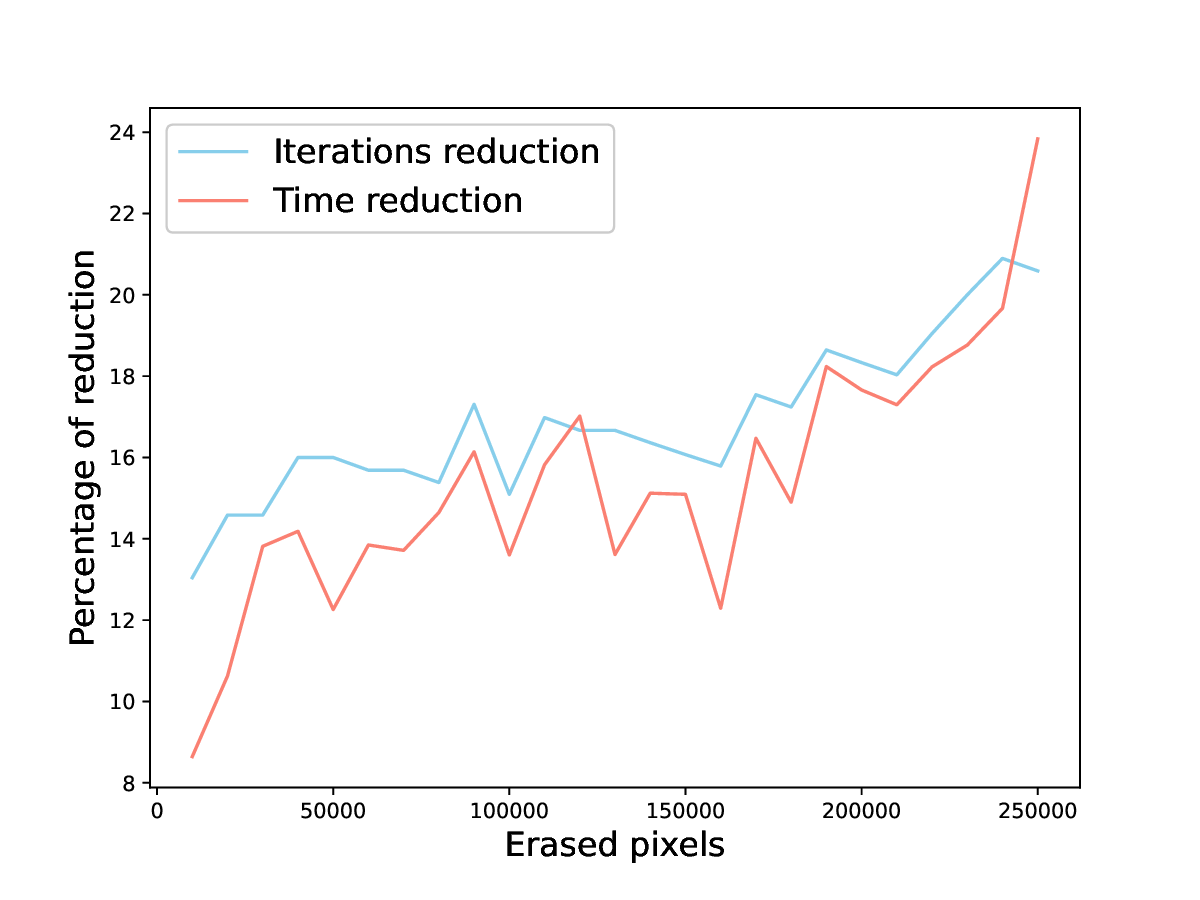}}
	\caption{Number of iterations (left), execution time (center) and percentage of reduction (right) in terms of the number of erased pixels, with step size $\rho=1$ and relaxation parameter $\lambda_k\equiv 1$, for a tolerance of $10^{-3}$.}
	\label{fig:compare_erased}
\end{figure}

{\bf Comparison in terms of the step size.} Both algorithms are tested for the same image with $250000$ randomly erased pixels for $\lambda_k\equiv 1$ and different values of the step size $\rho$. For the inertial version, the constant $\alpha$ in \eqref{seq:alphaInpainting} is adapted accordingly. The results are reported in Table~\ref{tab:gammaTest3Op} and depicted graphically in  Figure \ref{fig:compareGammas}. The percentage of reduction is noticeably higher for lower values of $\rho$ (always above 20\% when $\rho\le 1$). This is to be expected, since larger values of $\rho$ require lower values of $\alpha$, which limits the effect of inertia.

\begin{table}[h]
	\footnotesize
	\centering
	\begin{tabular}{ccccc}
		\toprule
		&  \multicolumn{2}{c}{Original algorithm} & \multicolumn{2}{c}{Inertial algorithm}\\
		$\rho$ & Time (s)  & Iterations  & Time (s) & Iterations \\
		\midrule
		0.1 & 119.80 & 524 & 70.04 & 301 \\
		0.2 & 64.25 & 281 & 39.28 & 169 \\
		0.3 & 44.61 & 195 & 28.55 & 122 \\
		0.4 & 34.88 & 150 & 22.67 & 98 \\
		0.5 & 28.20 & 123 & 19.80 & 83 \\
		0.6 & 23.90 & 104 & 17.17 & 73 \\
		0.7 & 21.13 & 91 & 15.46 & 66 \\
		0.8 & 18.46 & 81 & 14.08 & 61 \\
		0.9 & 16.74 & 74 & 13.68 & 58 \\
		1.0 & 15.81 & 69 & 13.25 & 56 \\
		1.1 & 14.87 & 65 & 12.94 & 56 \\
		1.2 & 14.60 & 64 & 13.24 & 56 \\
		1.3 & 14.34 & 63 & 13.23 & 57 \\
		1.4 & 14.67 & 64 & 13.39 & 58 \\
		1.5 & 14.55 & 64 & 13.90 & 60 \\
		\bottomrule
	\end{tabular}
	\caption{Execution time and number of iterations in terms of the step size $\rho$.}
	\label{tab:gammaTest3Op}
\end{table}

\begin{figure}[htb]
	\centering
	\subfloat{\includegraphics[width=0.3\textwidth,height=0.2\textwidth]{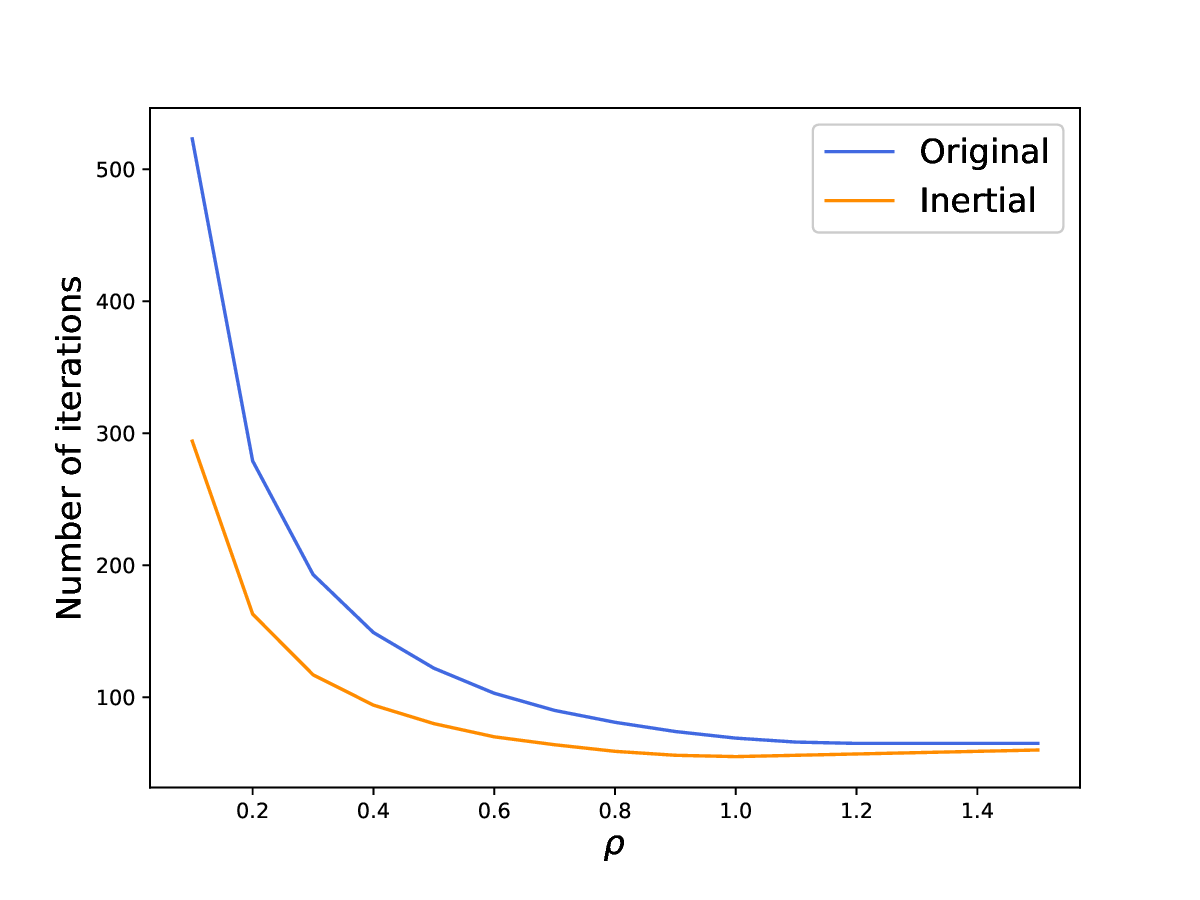}}
	\subfloat{\includegraphics[width=0.3\textwidth,height=0.2\textwidth]{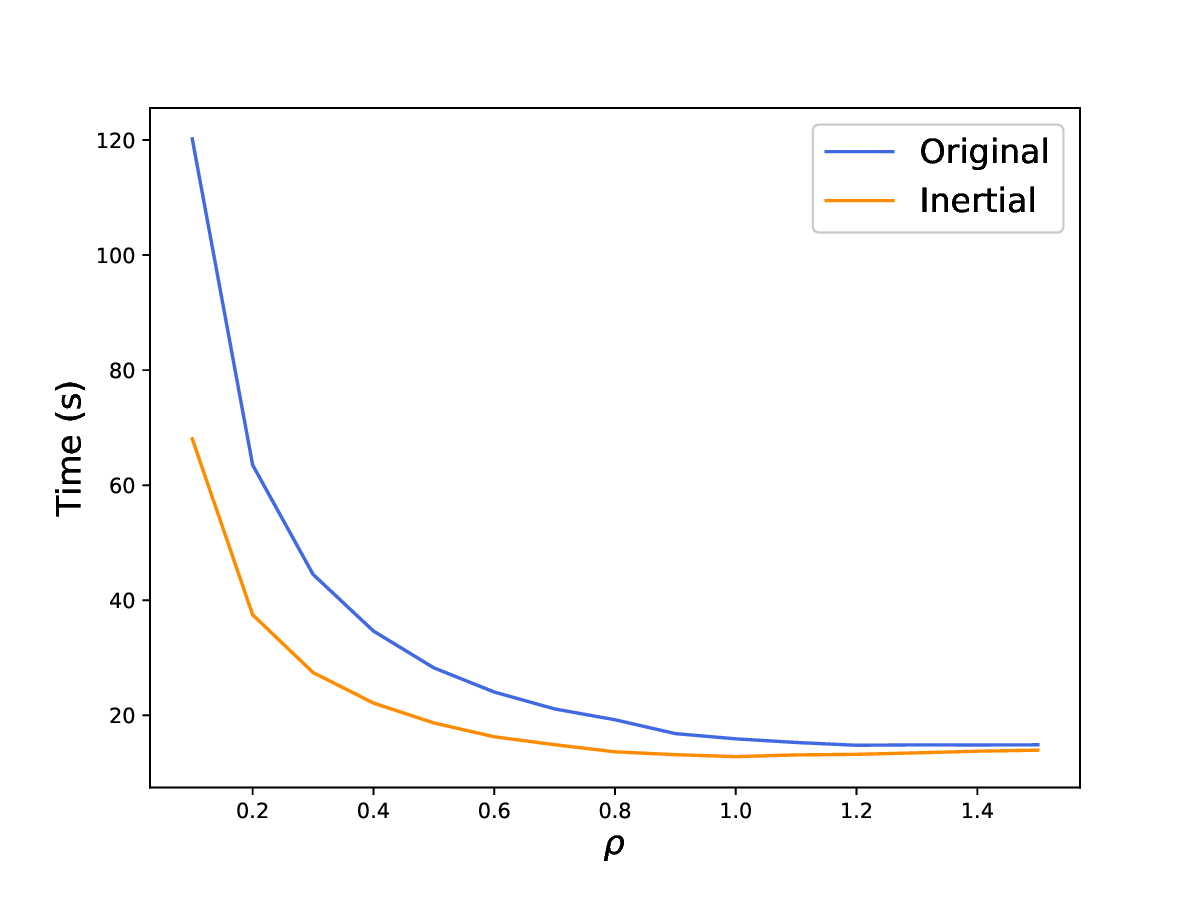}}
	\subfloat{\includegraphics[width=0.3\textwidth,height=0.2\textwidth]{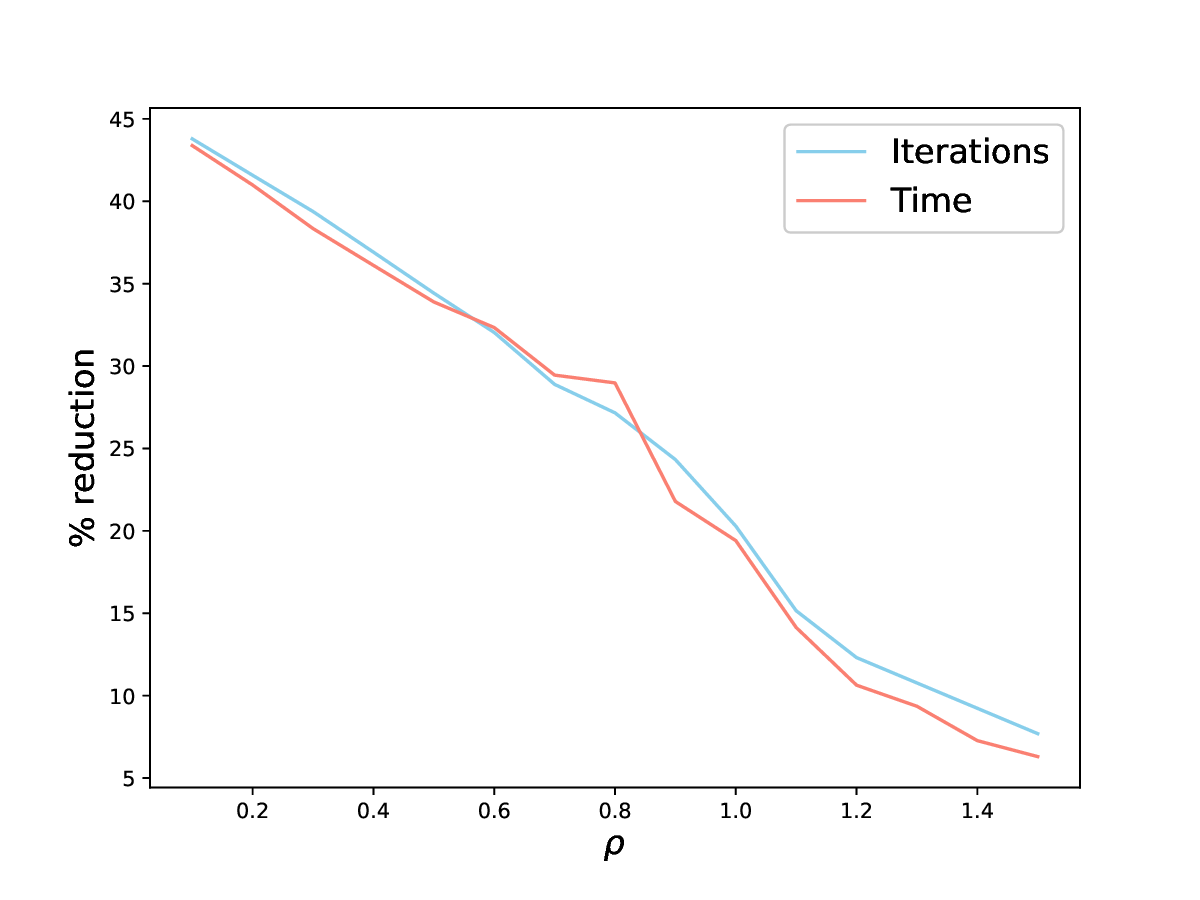}}
	\caption{Number of iterations (left), execution time (center) and percentage of reduction (right) in terms of the step size $\rho$.
	}
	\label{fig:compareGammas}
\end{figure}

{\bf Comparison in terms of the relaxation parameter.} Finally, we fix the value $\rho=1$, and compare the performance of the two methods for different values of the relaxation parameter $\lambda$, which, as before, limit the possible range for the inertial parameter $\alpha$ in view of condition \eqref{E:H2_constant_gamma}. The results are presented in Table \ref{tab:lambdaTest3Op}, and shown graphically in Figure \ref{fig:comparelambdas}. As with the step size, the reduction is greater for lower values of $\lambda$, which is consistent with the loss of the inertial character imposed by condition \eqref{E:H2_constant_gamma}. Nevertheless, observe that over-relaxing with $\lambda =1.2$ or $\lambda =1.4$ gives better results (both in number of iterations and execution time) than keeping $\lambda$ in a neighborhood of $1$.\\

\begin{table}[h]
	\footnotesize
	\centering
	\begin{tabular}{ccccc}
		\toprule
		&  \multicolumn{2}{c}{Original algorithm} & \multicolumn{2}{c}{Inertial algorithm}\\
		$\lambda$ & Time (s)  & Iterations  & Time (s) & Iterations \\
		\midrule
		0.6 & 24.47 & 108 & 13.57 & 56 \\
		0.7 & 21.28 & 94 & 11.65 & 51 \\
		0.8 & 18.67 & 83 & 12.64 & 55 \\
		0.9 & 16.94 & 75 & 12.76 & 56 \\
		1.0 & 15.52 & 69 & 12.76 & 56 \\
		1.1 & 14.28 & 63 & 12.51 & 55 \\
		1.2 & 13.35 & 59 & 12.53 & 54 \\
		1.3 & 12.52 & 55 & 11.90 & 52 \\
		1.4 & 12.04 & 52 & 11.71 & 51 \\
		\bottomrule
	\end{tabular}
	\caption{Execution time and number of iterations for different values of $\lambda$.}
	\label{tab:lambdaTest3Op}
\end{table}

\begin{figure}[h]
	\centering
	\subfloat{\includegraphics[width=0.3\textwidth,height=0.2\textwidth]{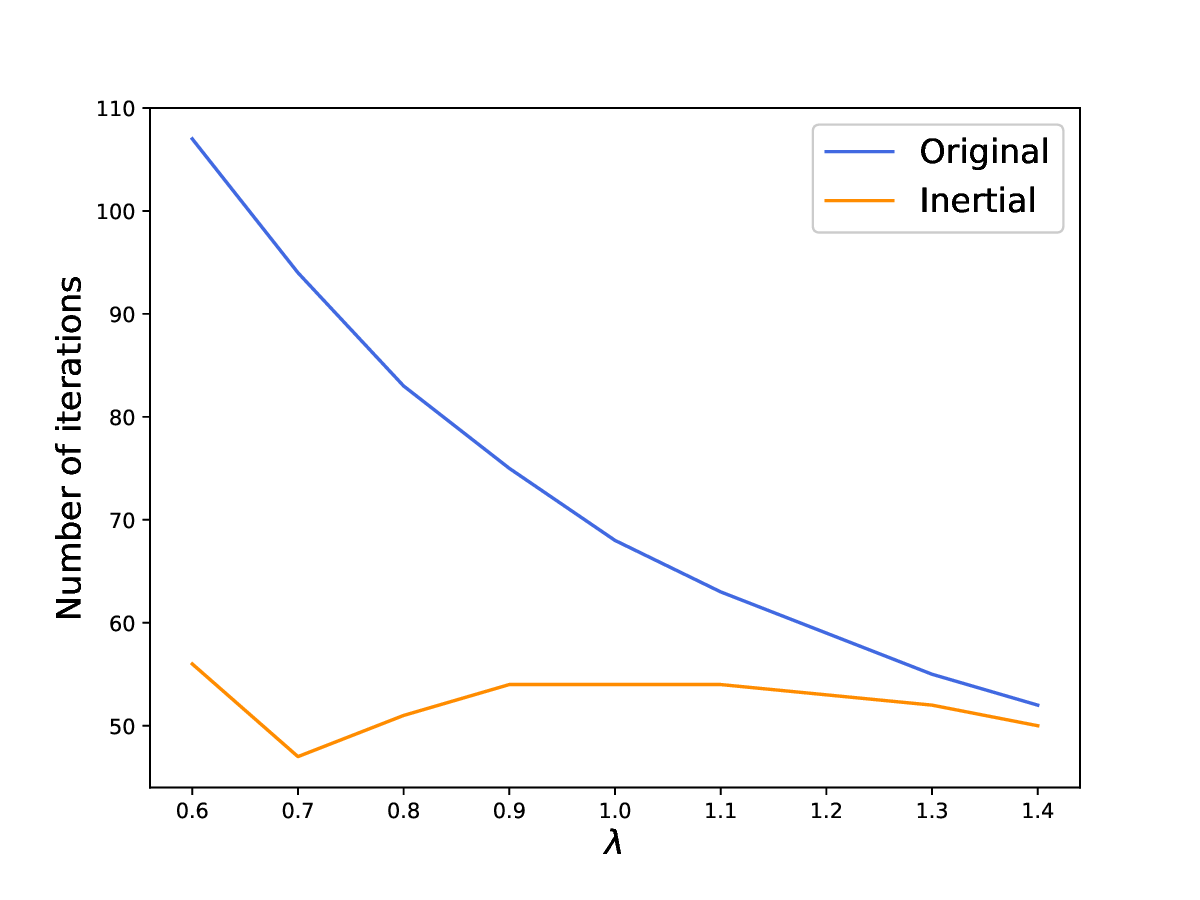}}
	\subfloat{\includegraphics[width=0.3\textwidth,height=0.2\textwidth]{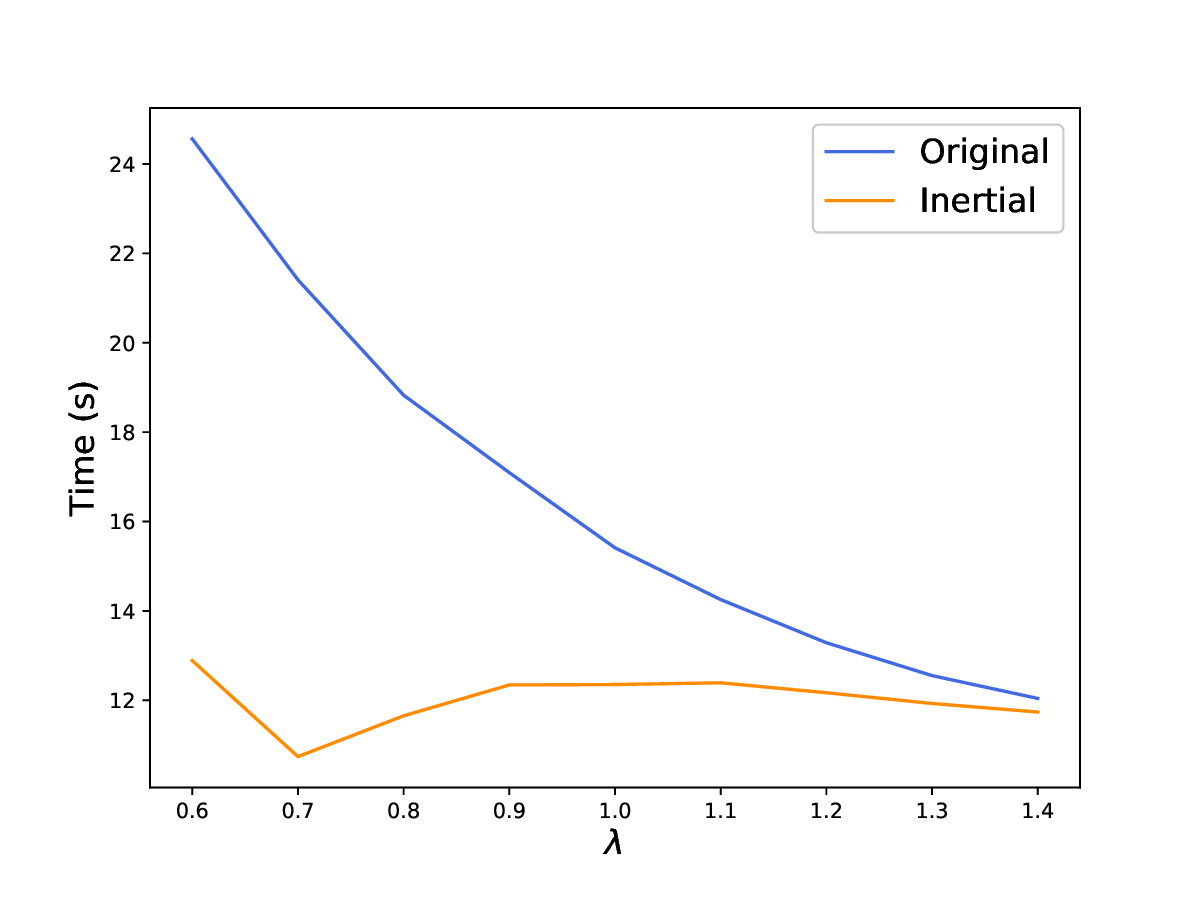}}
	\subfloat{\includegraphics[width=0.3\textwidth,height=0.2\textwidth]{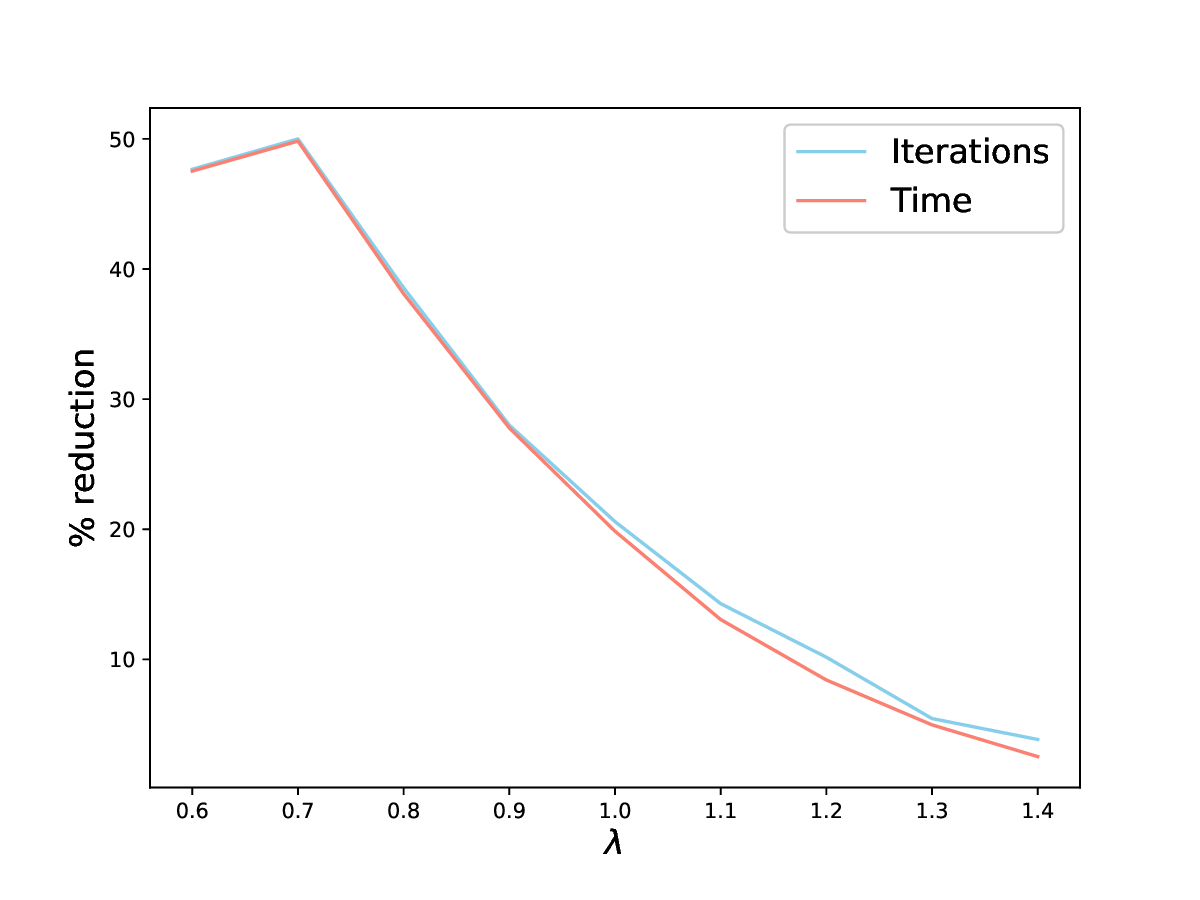}}
	\caption{Number of iterations (left), execution time (center) and percentage of reduction (right) in terms of the relaxation parameter $\lambda$.}
	\label{fig:comparelambdas}
\end{figure}

The evolution of the function values, the distance to the limit and the residuals are shown (in logarithmic scale) in Figure \ref{fig:3op_norm_values} for 250000 erased pixels, using $\rho=1$ and $\lambda_k\equiv 1$. As in the previous example, the sequence $k\norm{z_k - Tz_k}^2$ tends to zero, allowing us to conjecture again an asymptotic rate of $o(1/k)$. Finally, Figure \ref{fig:edificios} shows the original, corrupted (with 250000 erased pixels) and recovered images.\footnote{For the sake of a fair visual comparison, we follow the implementation used in \cite{davis2017three}, as described in \href{https://damek.github.io/ThreeOperators.html}{https://damek.github.io/ThreeOperators.html}, which differs slightly from the description given in Section \ref{section:3op} in that it contains a Bregman update.} 

\begin{figure}[!htbp]
	\centering
	\includegraphics[scale=0.3]{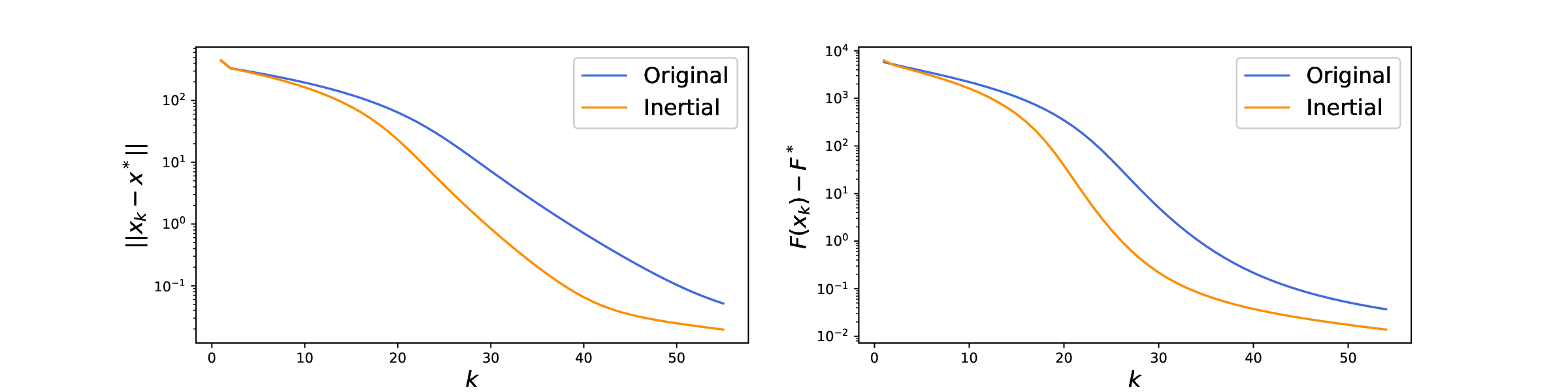}
	\includegraphics[scale=0.3]{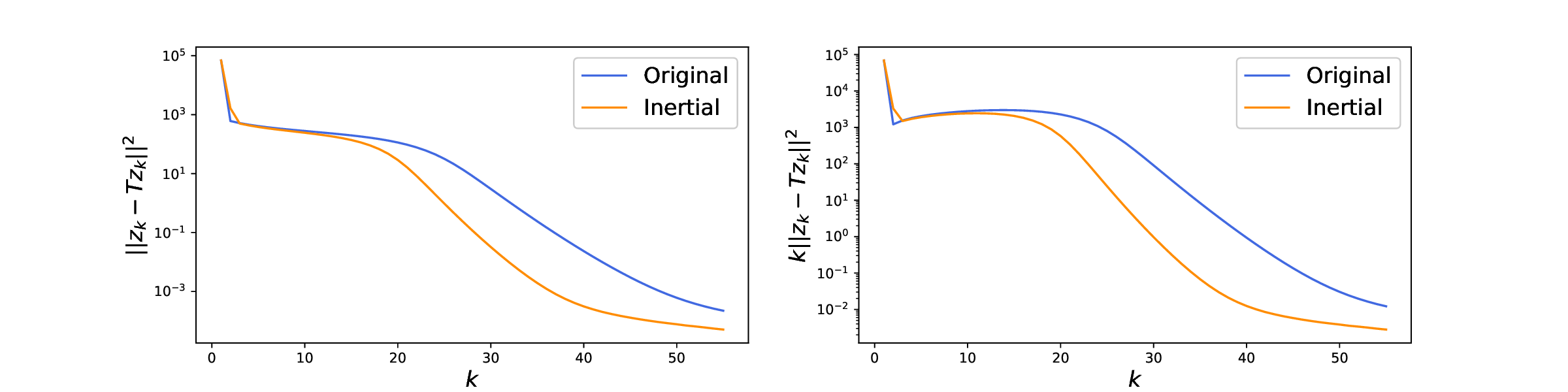}
	\caption{Evolution to the distance to the computed solution (top left), objective function values (top right), residuals $\norm{z_k - Tz_k}^2$ (bottom left) and $k\norm{z_k - Tz_k}^2$ (bottom right), for 250000 erased pixels using $\rho=1$ and $\lambda_k\equiv 1$.}
	\label{fig:3op_norm_values} 
\end{figure}

\begin{figure}[htbp]
	\centering
	\subfloat[Original image]{\includegraphics[width=0.22\textwidth]{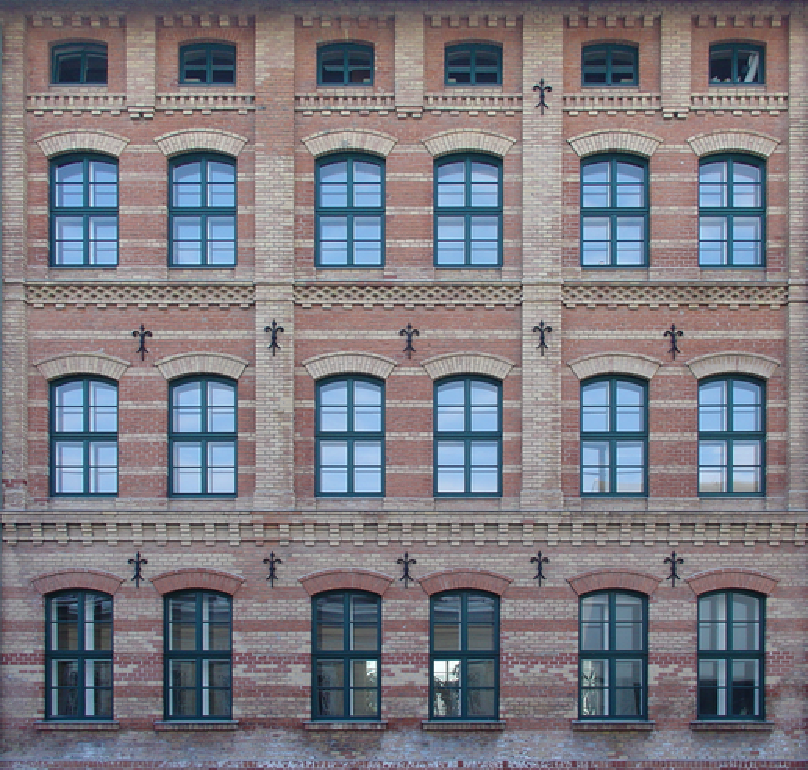} \label{fig:inpainting_original}}    
	\subfloat[Corrupted image]{\includegraphics[width=0.22\textwidth]{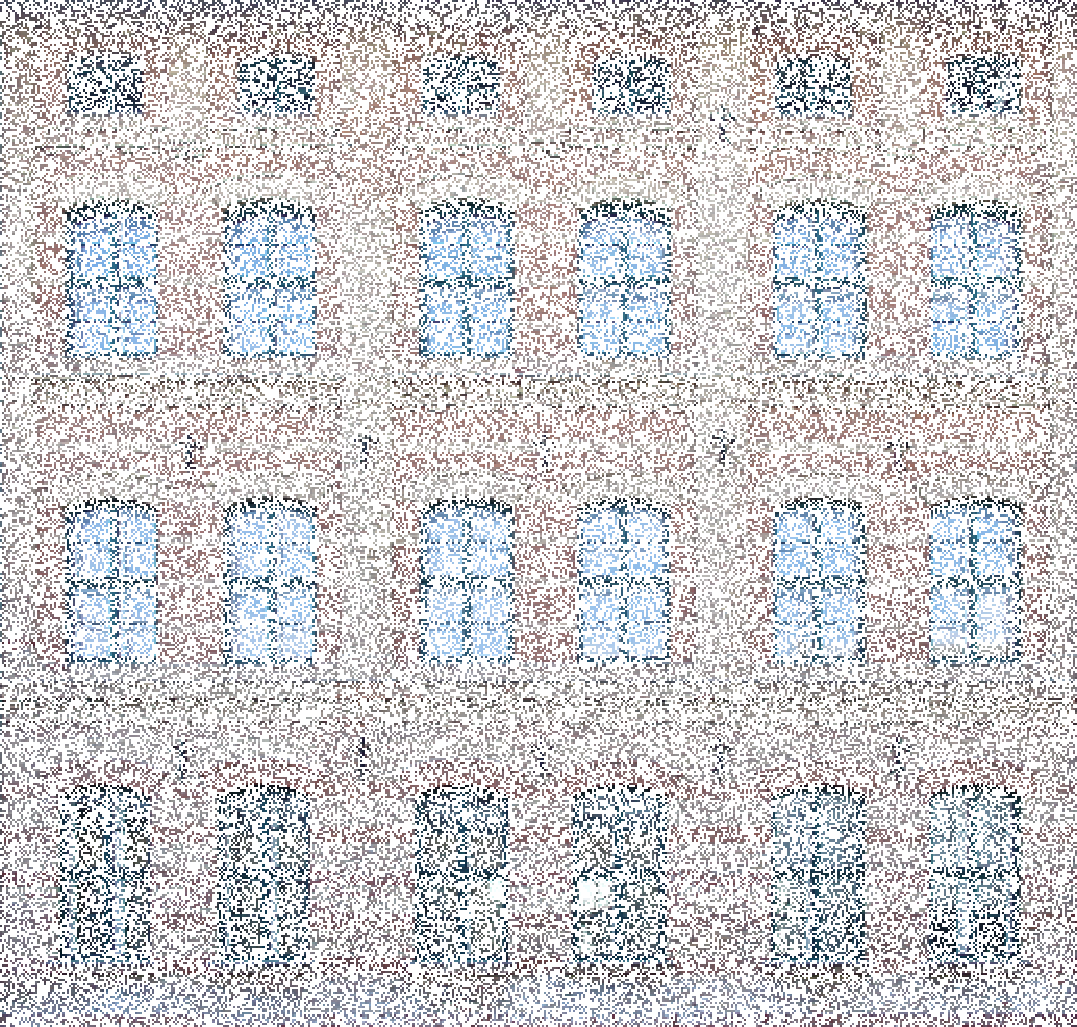}\label{fig:inpainting_erased}}  \hspace{0.1pt} 
	\subfloat[Recovered without inertia]{\includegraphics[width=0.22\textwidth]{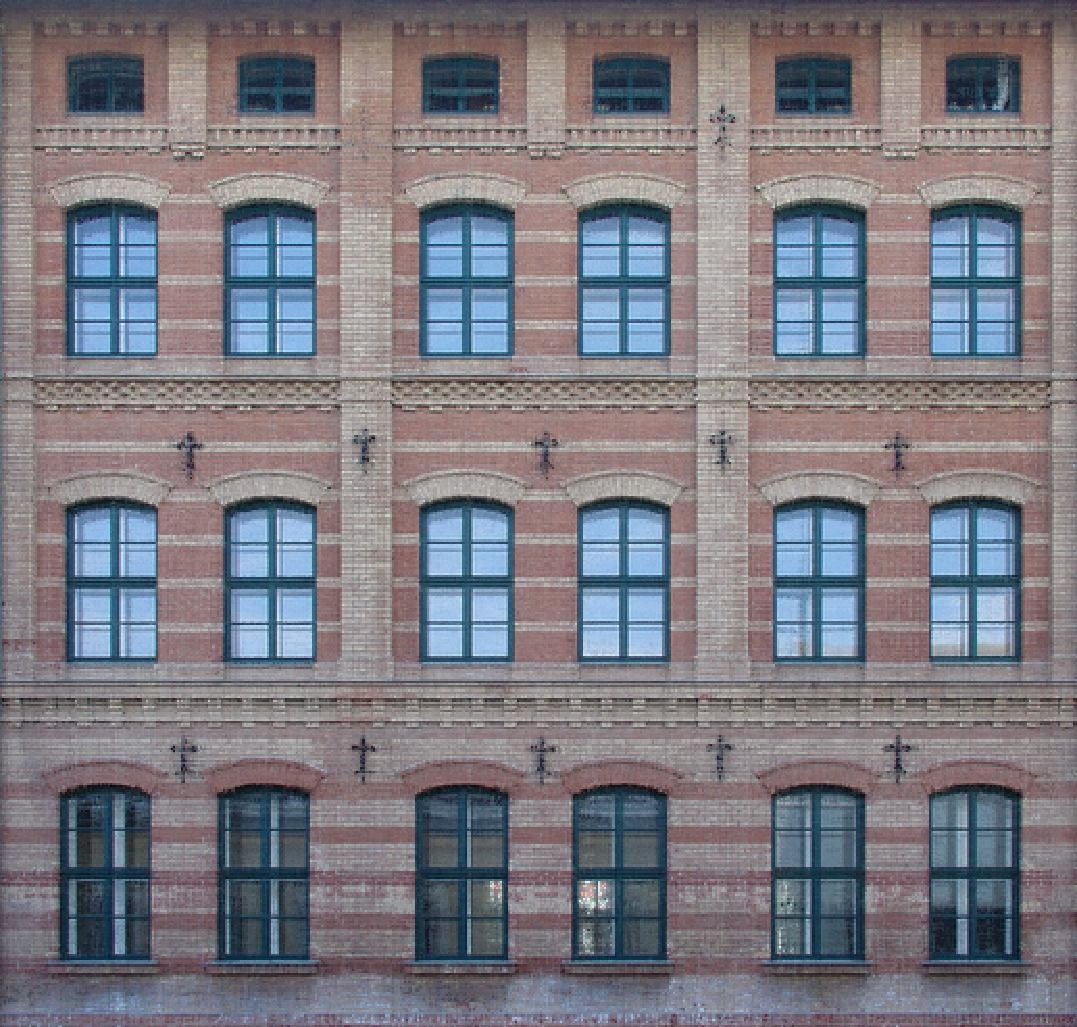}\label{fig:inpainting_recovered_orig}} 
	\hspace{0.1pt}
	\subfloat[Recovered with inertia]{\includegraphics[width=0.22\textwidth]{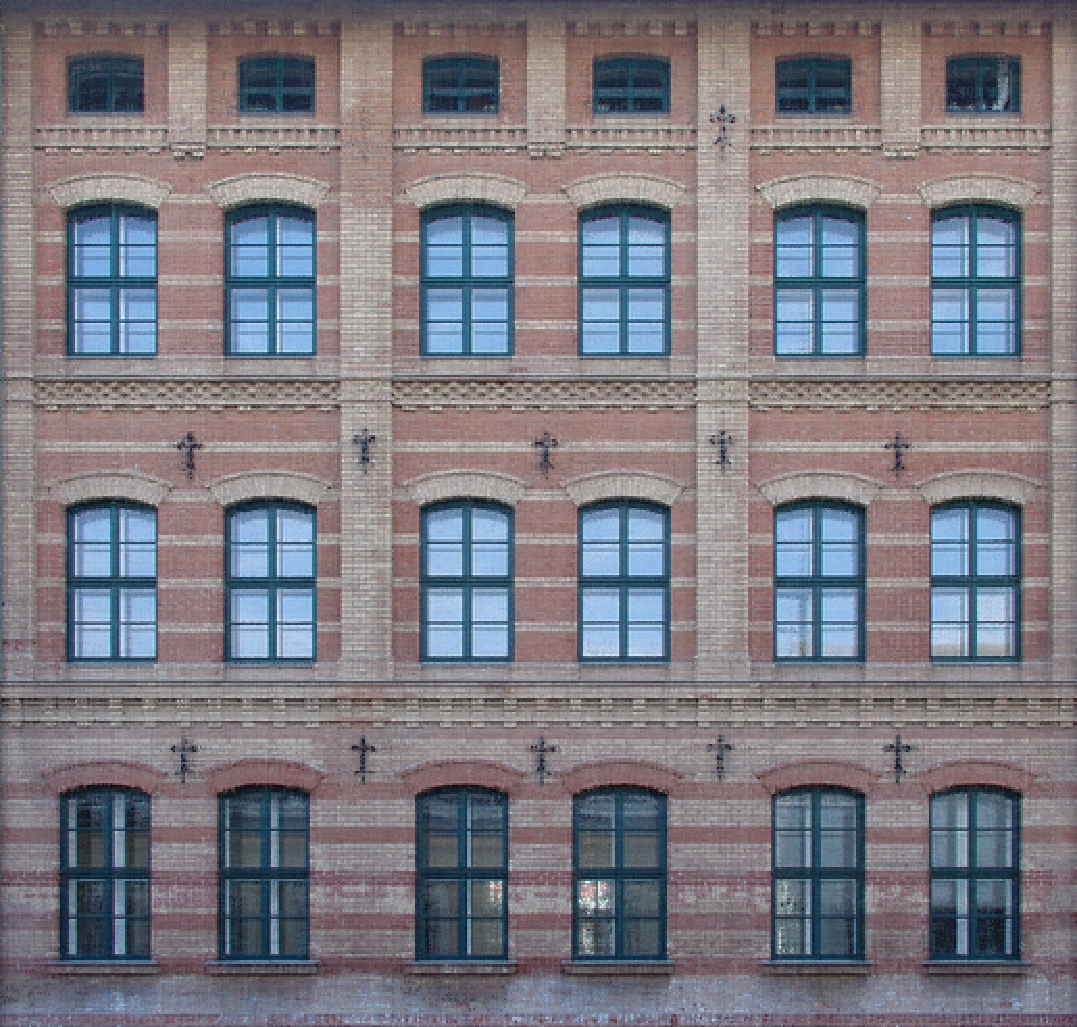}\label{fig:inpainting_recovered_iner}}
	\caption{Original image (a), corrupted image with 250000 randomly erased pixels (b), images recovered without inertia (c), and with inertia (d).}
	\label{fig:edificios}
\end{figure}

\if{ 
	\underline{\textbf{Bregman Iteration:}} As this experiment is a replication of one of the numerical simulations performed in \cite{davis2017three}, it is expected to obtain similar results. Considering the codes of Damek Davis and Wotao Yin\footnote{Scripts obtained from \href{https://damek.github.io/ThreeOperators.html}{https://damek.github.io/ThreeOperators.html}.} and the results presented in the publication, it can be seen that the obtained image is better than the one obtained in this work. This is because in that work, the coding of the Algorithm uses a Bregman update for the observation $Y$. That is, considering an observation $Y_k$. Then, the function $h$ changes at each iteration and also the operator $\nabla h$. This update improves the quality of the reconstructed image, but in this work that step was ignored for the purpose of defining the algorithm with the structure on \eqref{E:Algorithm}. Also, the aim of this experiment is to contrast the performance of the Inertial setting against the original in terms of the number of iterations and time, the modeling of the problem and generating the best possible image is beyond the scope of this paper. For a more  detailed description of the Bregman iterations and the modeling of the in painting problem, see \cite{liu2012tensor} and \cite{ma2011fixed}.   \\
	Figure \ref{fig:comp_Bregman} shows the difference between the images obtained by the Algorithm with and without Bregman update. The first experiment showed in this section is performed again, but with the algorithms using the Bregman update. Figure \ref{fig:Bregman_compare_erased} shows that using this heuristic, the inertial algorithm also outperforms the original version. 
	
	\begin{figure}[htbp]
		\centering
		\subfloat[Algorithm \eqref{algorithm:threeoperators} with $\alpha_k \equiv 0$. ]{\includegraphics[width=0.24\textwidth]{images/recovered_original.eps} }
		\subfloat[Algorithm \eqref{algorithm:threeoperators} with $\alpha_k \equiv 0$ and Bregman update.]{\includegraphics[width=0.24\textwidth]{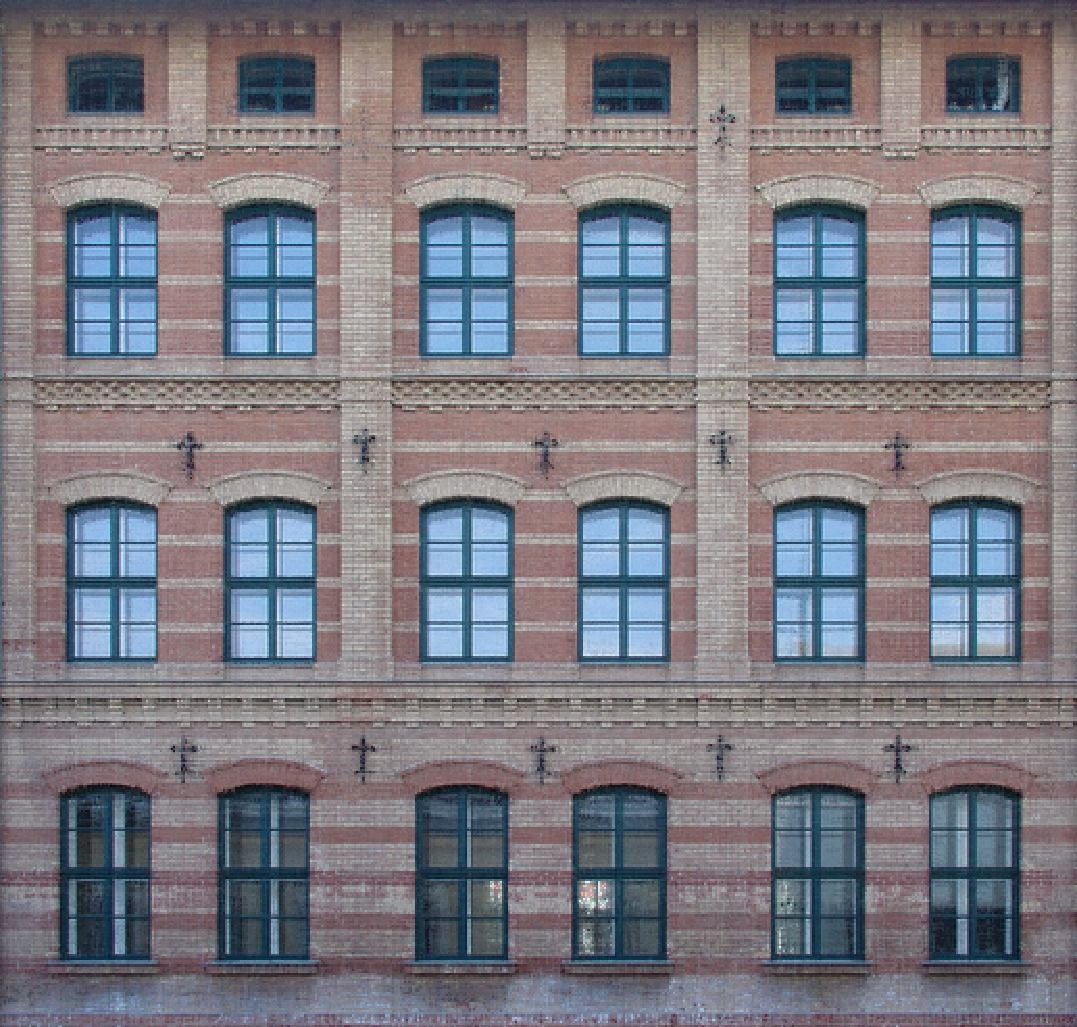}}  
		\subfloat[Inertial Algorithm \eqref{algorithm:threeoperators} . ]{\includegraphics[width=0.24\textwidth]{images/recovered_inertial.eps} }
		\subfloat[Inertial Algorithm \eqref{algorithm:threeoperators} and Bregman update.]{\includegraphics[width=0.24\textwidth]{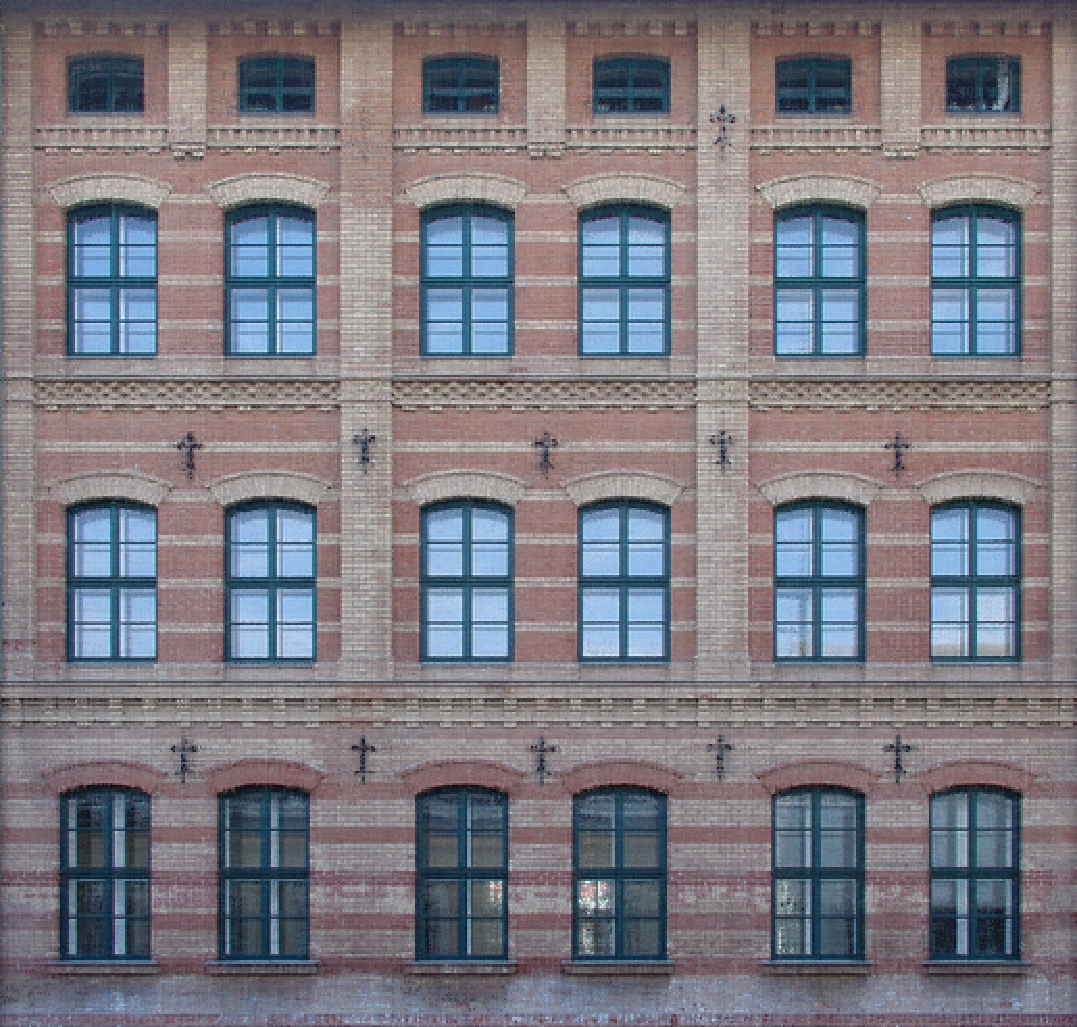}}
		\caption{Comparison of the recovered image obtained in this work vs. the one obtained in \cite{davis2017three}}
		
		\label{fig:comp_Bregman}
	\end{figure}
	
	\begin{figure}[htbp]
		\centering
		\subfloat[Number of iterations]{\includegraphics[width=0.33\textwidth]{images/Bregman_erased-image_iter.eps}}
		\subfloat[Time in seconds]{\includegraphics[width=0.33\textwidth]{images/Bregman_erased-image_time.eps}}
		\subfloat[Percentage of reduction in iterations performed by the inertial algorithm. ]{\includegraphics[width=0.33\textwidth]{images/Bregman_erased-image_reduction.eps}}
		\caption{Comparison between the original and the inertial algorithm for the three operators splitting scheme, using a tolerance of $\varepsilon=10^{-3}$, $\rho=1$, $\lambda=1$ for different amounts of erased pixels and using the Bregman update.}
		\label{fig:Bregman_compare_erased}
	\end{figure}

	\clearpage

}\fi

	The datasets generated during and/or analysed during the current study are available from the corresponding author upon request.
	\bibliographystyle{abbrv}
	\bibliography{referencias}
	
\end{document}